\documentclass[a4paper]{amsart}
\usepackage{color}
\usepackage[latin1]{inputenc}
\usepackage[T1]{fontenc}
\usepackage{amsfonts}
\usepackage{amssymb}
\usepackage{amsmath}
\usepackage{amsthm}
\usepackage{graphicx,type1cm,eso-pic,color}
\usepackage{pstricks,pst-plot,pstricks-add}
\usepackage{float}
\newtheorem{theorem}{Theorem}[section]
\newtheorem{lemma}[theorem]{Lemma}
\newtheorem{proposition}[theorem]{Proposition}
\newtheorem{corollary}[theorem]{Corollary}
\newtheorem{definition}[theorem]{Definition}
\newtheorem{assumption}[theorem]{Assumption}
\begin{document}
\setlength\arraycolsep{2pt}
\title{Singular Equations Driven by an Additive Noise and Applications}
\author{Nicolas MARIE}
\address{Laboratoire Modal'X Universit\'e Paris-Ouest 92000 Nanterre}
\email{nmarie@u-paris10.fr}
\address{Laboratoire ISTI, ESME Sudria, Paris, 75015, France}
\email{marie@esme.fr}
\date{}
\maketitle
%

% Abstract.

%
\begin{abstract}
In the pathwise stochastic calculus framework, the paper deals with the general study of equations driven by an additive Gaussian noise, with a drift function having an infinite limit at point zero.
\\
An ergodic theorem and the convergence of the implicit Euler scheme are proved. The Malliavin calculus is used to study the absolute continuity of the distribution of the solution. An estimation procedure of the parameters of the random component of the model is provided.
\\
The properties are transferred on a class of singular stochastic differential equations driven by a multiplicative noise. A fractional Heston model is introduced.
\end{abstract}
\tableofcontents
%

% Section : Introduction.

%
\section{Introduction}
\noindent
Let $B := (B_t)_{t\in\mathbb R_+}$ be a centered stochastic process with locally $\alpha$-H\"older continuous paths, and consider the stochastic differential equation
\begin{equation}\label{main_equation}
X_t =
x_0 +
\int_{0}^{t}b(X_s)ds +
\sigma B_t
\end{equation}
where $\alpha\in ]0,1[$, $x_0\in\textrm I$, $\textrm I\subset\mathbb R$ is an interval, $\sigma\in\mathbb R^* :=\mathbb R -\{0\}$ and $b :\textrm I\rightarrow\mathbb R$ is a $[1/\alpha] + 1$ times continuously differentiable function.
\\
\\
Assume that $\textrm I =\mathbb R$ and $b$ is everywhere differentiable with bounded derivatives. Then, Equation (\ref{main_equation}) has a unique (pathwise) solution defined on $\mathbb R_+$ with locally $\alpha$-H\"older continuous paths (see Friz and Victoir \cite{FV10}, sections 10.3 and 10.7).
\\
If in addition $B$ is a fractional Brownian motion (see Nualart \cite{NUALART06}, Chapter 5), the probabilistic and statistical properties of the solution of Equation (\ref{main_equation}) have been deeply studied by several authors (see Hairer \cite{HAIRER05}, Tudor and Viens \cite{TV07}, Neuenkirch and Tindel \cite{NT14}, etc.).
\\
\\
Throughout the paper, $\textrm I = ]0,\infty[$ and
\begin{displaymath}
\lim_{x\rightarrow 0^+}
b(x) =\infty.
\end{displaymath}
The existence and the uniqueness of the solution of Equation (\ref{main_equation}), and the absolute continuity of its distribution for a fractional Brownian signal of Hurst parameter belonging to $]1/2,1[$ have been already studied in Hu, Nualart and Song \cite{HNS08}.
\\
The current paper deals with a general study of Equation (\ref{main_equation}) in the pathwise stochastic calculus framework (see Lyons \cite{LYONS98}, Lyons and Qian \cite{LQ02}, Gubinelli and Lejay \cite{GL09}, Lejay \cite{LEJAY10}, Friz and Victoir \cite{FV10}, etc.) under the following Assumption.
%

% Assumption : Assumption on the drift function.

%
\begin{assumption}\label{drift_assumption}
\white .\black
\begin{enumerate}
 \item The function $b$ is $[1/\alpha] + 1$ times continuously differentiable on $]0,\infty[$ and has bounded derivatives on $[\varepsilon,\infty[$ for every $\varepsilon > 0$.
 \item There exists a constant $K > 0$ such that :
 \begin{displaymath}
 \forall x > 0
 \textrm{$,$ }
 \dot b(x) < -K.
 \end{displaymath}
 \item There exists a constant $R > 0$ such that :
 \begin{displaymath}
 \forall x > 0
 \textrm{$,$ }
 b(x) > -Rx.
 \end{displaymath}
 \item For every $C > 0$,
 \begin{displaymath}
 \int_{0}^{T}
 b(Ct^{\alpha})dt =\infty
 \textrm{ $;$ }
 \forall T > 0
 \end{displaymath}
 or
 \begin{displaymath}
 \lim_{T\rightarrow 0^+}
 \frac{1}{T^{\alpha}}\int_{0}^{T}
 b(Ct^{\alpha})dt =\infty.
 \end{displaymath}
\end{enumerate}
\end{assumption}
\noindent
The second section is devoted to deterministic properties of Equation (\ref{main_equation}) : the global existence and the uniqueness of the solution, the regularity of the It\^o map, the convergence of the implicit Euler scheme and some estimates.
\\
\\
The third section is devoted to probabilistic and statistical properties of the solution $X(x_0)$ of Equation (\ref{main_equation}), obtained via its deterministic properties proved at Section 2 and various additional conditions on the signal $B$. In order to ensure the integrability of estimates, $B$ is a Gaussian process in the major part of Section 3.
\\
Subsection 3.1 deals with the ergodicity of $X(x_0)$, studied in the random dynamical systems framework (see Arnold \cite{ARNOLD98}). By assuming that $B$ is a fractional Brownian motion, the existence of an attracting stationary solution of Equation (\ref{main_equation}) and an ergodic theorem are proved.
\\
Subsection 3.2 deals with applications of the Malliavin calculus (see Nualart \cite{NUALART06}) to the absolute continuity of the distribution of $X_t(x_0)$ for every $t\in ]0,T]$. Via Nourdin and Viens \cite{NV09}, a density with a suitable expression is provided.
\\
Subsection 3.3 deals with the integrability and the convergence of the implicit Euler scheme. A rate of convergence is provided.
\\
Subsection 3.4 deals with a relationship between $X(x_0)$ and an Ornstein-Uhlenbeck process. By assuming that $B$ is a fractional Brownian motion of Hurst parameter $H\in ]1/2,1[$, an estimation procedure of $(H,\sigma)$ is provided by using Melichov \cite{MELICHOV11}, Brouste and Iacus \cite{BI13}, and Berzin and Le\'on \cite{BL08}. On the fractional Ornstein-Uhlenbeck process, see Cheridito et al. \cite{CKM03} and Garrido-Atienza et al. \cite{GAKN09}.
\\
\\
The fourth section is devoted to the transfer of the properties established at sections 2 and 3 on a class of singular stochastic differential equations driven by a multiplicative noise. In particular, it covers and completes Marie \cite{MARIE14} on a generalized Cox-Ingersoll-Ross model.
\\
Subsection 4.2 deals with a Heston model (see Heston \cite{HESTON93}) in which the volatility is modeled by a fractional Cox-Ingersoll-Ross equation in order to take benefits of the long memory and of the regularity of the paths of the fractional Brownian motion as in Comte, Coutin and Renault \cite{CCR12}.
\\
\\
\textbf{Notations.} Let $\textrm J\subset\mathbb R$ be a compact interval.
\begin{itemize}
 \item The space $C^0(\textrm J,\mathbb R)$ of the continuous functions from $\textrm J$ into $\mathbb R$ is equipped with the uniform norm $\|.\|_{\infty,\textrm J}$ defined by :
 \begin{displaymath}
 \|x\|_{\infty,\textrm J} :=\sup_{t\in\textrm J}|x_t|
 \end{displaymath}
 for every $x\in C^0(\textrm J,\mathbb R)$. If $\textrm J = [0,T]$ with $T > 0$, the uniform norm is denoted by $\|.\|_{\infty,T}$.
 \item The space $C^{\alpha}(\textrm J,\mathbb R)$ of the $\alpha$-H\"older continuous functions from $\textrm J$ into $\mathbb R$ is equipped with $\|.\|_{\infty,T}$, or with the $\alpha$-H\"older norm $\|.\|_{\alpha,\textrm J}$ defined by :
 \begin{displaymath}
 \|x\|_{\alpha,\textrm J} :=\sup_{(s,t)\in\textrm J^2\textrm{ $:$ }s < t}
 \frac{|x_t - x_s|}{|t - s|^{\alpha}}
 \end{displaymath}
 for every $x\in C^{\alpha}(\textrm J,\mathbb R)$. If $\textrm J = [0,T]$ with $T > 0$, the $\alpha$-H\"older norm is denoted by $\|.\|_{\alpha,T}$.
 \item The space $C^0(\mathbb R_+,\mathbb R)$ (resp. $C^{\alpha}(\mathbb R_+,\mathbb R)$) of the continuous functions from $\mathbb R_+$ into $\mathbb R$ (resp. of the locally $\alpha$-H\"older continuous functions from $\mathbb R_+$ into $\mathbb R$) is equipped with the compact-open topology (i.e. for every sequence $(f_n)_{n\in\mathbb N}$ of $C^0(\mathbb R_+,\mathbb R)$, $f_n\rightarrow f$ when $n\rightarrow\infty$ for the compact-open topology if and only if, for every compact subset $\textrm K$ of $\mathbb R_+$,
 \begin{displaymath}
 \lim_{n\rightarrow\infty}
 \|f_n - f\|_{\infty,\textrm K} = 0).
 \end{displaymath}
\end{itemize}
%

% Section : Deterministic properties of the solution.

%
\section{Deterministic properties of the solution}
\noindent
The section deals with the global existence and the uniqueness of the solution of Equation (\ref{main_equation}), the regularity of the It\^o map, the convergence of the implicit Euler scheme and some estimates.
\\
\\
First of all, some examples of drift functions satisfying Assumption \ref{drift_assumption} are provided.
\\
\\
\textbf{Examples.} Consider $u,v,w,\gamma,\lambda,\mu > 0$.
\begin{itemize}
 \item Put $b_1(x) := u(vx^{-\gamma} - wx)$ for every $x > 0$. If $1 -\alpha <\alpha\gamma$, then $b_1$ satisfies Assumption \ref{drift_assumption}.
 \item Put $b_2(x) := u/(e^{vx^{\gamma}} - 1) - wx$ for every $x > 0$. If $1\leqslant\alpha\gamma$, then $b_2$ satisfies Assumption \ref{drift_assumption}.
 \item Put $b_{1}^{*}(x) := \lambda\sin(\mu x)$ for every $x > 0$. If $1 -\alpha <\alpha\gamma$ (resp. $1\leqslant\alpha\gamma$) and $\lambda\mu < uw$ (resp. $\lambda\mu < w$), then $b_1 + b_{1}^{*}$ (resp $b_2 + b_{1}^{*}$) satisfies Assumption \ref{drift_assumption}.
 \item Put $b_{2}^{*}(x) := \lambda\log(\mu x)$ for every $x > 0$. If $1 -\alpha <\alpha\gamma$ (resp. $1\leqslant\alpha\gamma$), then $b_1 + b_{2}^{*}$ (resp $b_2 + b_{2}^{*}$) satisfies Assumption \ref{drift_assumption}.
\end{itemize}
%

% Subsection : Existence and uniqueness of the solution.

%
\subsection{Existence and uniqueness of the solution}
The subsection deals with the global existence, the uniqueness and an estimate of the solution of Equation (\ref{main_equation}).
\\
\\
Consider the deterministic analog of Equation (\ref{main_equation}) :
\begin{equation}\label{main_equation_deterministic}
x_t = x_0 +
\int_{0}^{t}b(x_s)ds +\sigma w_t
\end{equation}
with $w\in C^{\alpha}(\mathbb R_+,\mathbb R)$.
\\
\\
By Assumption \ref{drift_assumption}.(1), Equation (\ref{main_equation_deterministic}) has a unique solution on $[0,T_0]$, where
\begin{displaymath}
T_0 :=
\inf\{t > 0 : x_t = 0\}
\end{displaymath}
with the convention $\inf(\emptyset) =\infty$.
%

% Proposition : Existence and uniqueness of the solution.

%
\begin{proposition}\label{existence_solution}
Under Assumption \ref{drift_assumption}, Equation (\ref{main_equation_deterministic}) has a unique $]0,\infty[$-valued solution on $\mathbb R_+$.
\end{proposition}
%

% Proof.

%
\begin{proof}
Assume that $T_0 <\infty$ and put $y := e^{R.}x$ on $[0,T_0]$. For every $t\in [0,T_0]$, by the rough change of variable formula (see Gubinelli and Lejay \cite{GL09}, Lemma 6) :
\begin{eqnarray}
 y_t & = & y_0 +
 \int_{0}^{t}
 Re^{Rs}x_sds +
 \int_{0}^{t}
 e^{Rs}dx_s
 \nonumber\\
 \label{auxiliary_equation_existence}
 & = &
 y_0 +
 \int_{0}^{t}
 b^R(s,y_s)ds +
 \sigma w_{t}^{R}
\end{eqnarray}
where
\begin{displaymath}
b^R(t,u) = Ru + e^{Rt}b(e^{-Rt}u)
\end{displaymath}
for every $u > 0$, and
\begin{displaymath}
w_{t}^{R} :=
\int_{0}^{t}e^{Rs}dw_s.
\end{displaymath}
For $t\in [0,T_0]$ arbitrarily chosen, by Equation (\ref{auxiliary_equation_existence}) :
\begin{displaymath}
y_t +\int_{t}^{T_0}b^R(s,y_s)ds =
\sigma(w_{t}^{R} - w_{T_0}^{R}).
\end{displaymath}
Then, since $w^R$ is $\alpha$-H\"older continuous on $[0,T_0]$ and $b^R(s,u) > 0$ for every $(s,u)\in\mathbb R_+\times ]0,\infty[$ by Assumption \ref{drift_assumption}.(3) :
\begin{eqnarray*}
 y_s & \leqslant &
 |\sigma|\|w^R\|_{\alpha,T_0}|s - T_0|^{\alpha}
 \textrm{ $;$ $\forall s\in [0,T_0]$ and}\\
 \int_{t}^{T_0}
 b^R(s,y_s)ds & \leqslant &
 |\sigma|\|w^R\|_{\alpha,T_0}|t - T_0|^{\alpha}.
\end{eqnarray*}
Since $b$ is strictly decreasing on $]0,\infty[$ by Assumption \ref{drift_assumption}.(2) :
\begin{eqnarray*}
 \int_{t}^{T_0}
 b^R(s,y_s)ds
 & \geqslant &
 \int_{t}^{T_0}
 b(e^{-Rs}y_s)ds\\
 & \geqslant &
 \int_{0}^{T_0 - t}b(\|w^R\|_{\alpha,T_0}s^{\alpha})ds.
\end{eqnarray*}
Therefore,
\begin{displaymath}
 \int_{0}^{T_0 - t}b(\|w^R\|_{\alpha,T_0}s^{\alpha})ds
 \leqslant
 |\sigma|\|w^R\|_{\alpha,T_0}(T_0 - t)^{\alpha}.
\end{displaymath}
However,
\begin{displaymath}
\int_{0}^{T_0 - t}b(\|w^R\|_{\alpha,T_0}s^{\alpha})ds =
\infty
\end{displaymath}
or
\begin{displaymath}
\lim_{t\rightarrow T_{0}^{-}}
\frac{1}{(T_0 - t)^{\alpha}}
\int_{0}^{T_0 - t}b(\|w^R\|_{\alpha,T_0}s^{\alpha})ds =
\infty
\end{displaymath}
by Assumption \ref{drift_assumption}.(4). That contradiction finishes the proof.
\end{proof}
%

% Proposition : Estimate of the solution.

%
\begin{proposition}\label{solution_estimate}
Under Assumption \ref{drift_assumption}, the solution $x$ of Equation (\ref{main_equation_deterministic}) \mbox{satisfies :}
\begin{displaymath}
\|x\|_{\infty,T}
\leqslant
x_0 + |b(x_0)|T +
2\sigma\|w\|_{\infty,T}
\end{displaymath}
for every $T > 0$.
\end{proposition}
%

% Proof.

%
\begin{proof}
Let $T > 0$ and $t\in [0,T]$ be arbitrarily chosen, and put
\begin{displaymath}
T_{x_0}(t) :=
\sup\{s\in [0,t] : x_t\leqslant x_0\}.
\end{displaymath}
If $T_{x_0}(t) = t$, then $0 < x_t\leqslant x_0$. Assume that $T_{x_0}(t) < t$. Then,
\begin{displaymath}
x_t = x_0 +
\int_{T_{x_0}(t)}^{t}b(x_s)ds +
\sigma[w_t - w_{T_{x_0}(t)}].
\end{displaymath}
By Assumption \ref{drift_assumption}.(2) :
\begin{eqnarray*}
 \int_{T_{x_0}(t)}^{t}
 b(x_s)ds & \leqslant &
 b(x_0)[t - T_{x_0}(t)]\\
 & \leqslant &
 |b(x_0)|t.
\end{eqnarray*}
Therefore,
\begin{displaymath}
0 < x_t
\leqslant
x_0 + |b(x_0)|T + 2\sigma\|w\|_{\infty,T}.
\end{displaymath}
That finishes the proof.
\end{proof}
\noindent
\textbf{Notation.} In the sequel, the solution of Equation (\ref{main_equation_deterministic}) with the initial condition $x_0 > 0$ and the driving signal $w\in C^{\alpha}(\mathbb R_+,\mathbb R)$ is denoted by $x(x_0,w)$. For every $T > 0$, the restriction of the It\^o map $x(.)$ to $]0,\infty[\times C^{\alpha}([0,T],\mathbb R)$ is also denoted by $x(.)$. Then,
\begin{displaymath}
x(x_0,w)|_{[0,T]} =
x(x_0,w|_{[0,T]})
\end{displaymath}
for every $x_0,T > 0$ and $w\in C^{\alpha}(\mathbb R_+,\mathbb R)$.
%

% Subsection : Regularity of the It map.

%
\subsection{Regularity of the It\^o map}
The two following propositions deal with the regularity of the It\^o map $x(.)$.
%

% Proposition : Lipschitz continuity of the It map.

%
\begin{proposition}\label{lipschitz_continuity_solution}
Under Assumption \ref{drift_assumption} :
\begin{displaymath}
\|x(x_{0}^{1},w^1) - x(x_{0}^{2},w^2)\|_{\infty,T}
\leqslant
|x_{0}^{1} - x_{0}^{2}| +
2\sigma\|w^1 - w^2\|_{\infty,T}
\end{displaymath}
for every $T > 0$, $x_{0}^{1},x_{0}^{2} > 0$ and $w^1,w^2\in C^{\alpha}([0,T],\mathbb R)$.
\end{proposition}
%

% Proof.

%
\begin{proof}
Consider $x_{0}^{1},x_{0}^{2} > 0$ and $w^1,w^2\in C^{\alpha}([0,T],\mathbb R)$ for $T > 0$ arbitrarily chosen.
\\
\\
Put $x^1 := x(x^1,w^1)$, $x^2 := x(x^2,w^2)$ and
\begin{displaymath}
T_{\textrm{c}} :=
\inf\{t\in [0,T] : x_{t}^{1} = x_{t}^{2}\}.
\end{displaymath}
Assume that $x_{0}^{1} > x_{0}^{2}$ without loss of generality. Since $x^1$ and $x^2$ are continuous on $\mathbb R_+$, $x_{s}^{1} > x_{s}^{2}$ for every $s\in [0,T_{\textrm{c}}]$. Since $b$ is strictly decreasing on $]0,\infty[$ by Assumption \ref{drift_assumption}.(2) :
\begin{displaymath}
b(x_{s}^{1}) - b(x_{s}^{2})
\leqslant 0
\end{displaymath}
for every $s\in [0,T_{\textrm{c}}]$. Then, for every $t\in [0,T_{\textrm{c}}]$,
\begin{eqnarray}
 |x_{t}^{1} - x_{t}^{2}| & = &
 x_{t}^{1} - x_{t}^{2}
 \nonumber\\
 & = &
 x_{0}^{1} - x_{0}^{2} +
 \int_{0}^{t}
 [b(x_{s}^{1}) - b(x_{s}^{2})]ds +
 \sigma(w_{t}^{1} - w_{t}^{2})
 \nonumber\\
 \label{lipschitz_continuity_estimate_1}
 & \leqslant &
 |x_{0}^{1} - x_{0}^{2}| +
 \sigma
 \|w^1 - w^2\|_{\infty,T}.
\end{eqnarray}
For $t\in [T_{\textrm{c}},T]$ arbitrarily chosen, put
\begin{displaymath}
T_{\textrm{c}}(t) :=
\sup\{s\in [T_{\textrm{c}},t] : x_{s}^{1} = x_{s}^{2}\}.
\end{displaymath}
Assume that $x_{t}^{1} > x_{t}^{2}$ without loss of generality. Since $x^1$ and $x^2$ are continuous on $\mathbb R_+$, $x_{s}^{1} > x_{s}^{2}$ for every $s\in [T_{\textrm{c}}(t),t]$. Since $b$ is strictly decreasing on $]0,\infty[$ by Assumption \ref{drift_assumption}.(2) :
\begin{displaymath}
b(x_{s}^{1}) - b(x_{s}^{2})
\leqslant 0
\end{displaymath}
for every $s\in [T_{\textrm{c}}(t),t]$. Then,
\begin{eqnarray}
 |x_{t}^{1} - x_{t}^{2}| & = &
 x_{t}^{1} - x_{t}^{2}
 \nonumber\\
 & = &
 \int_{T_{\textrm{c}}(t)}^{t}
 [b(x_{s}^{1}) - b(x_{s}^{2})]ds +
 \sigma(w_{t}^{1} - w_{t}^{2}) -
 \sigma[w_{T_{\textrm{c}}(t)}^{1} - w_{T_{\textrm{c}}(t)}^{2}]
 \nonumber\\
 \label{lipschitz_continuity_estimate_2}
 & \leqslant &
 2\sigma
 \|w^1 - w^2\|_{\infty,T}.
\end{eqnarray}
In conclusion, by inequalities (\ref{lipschitz_continuity_estimate_1}) and (\ref{lipschitz_continuity_estimate_2}) together :
\begin{displaymath}
\|x^1 - x^2\|_{\infty,T}
\leqslant
|x_{0}^{1} - x_{0}^{2}| +
2\sigma
\|w^1 - w^2\|_{\infty,T}.
\end{displaymath}
That finishes the proof.
\end{proof}
\noindent
\textbf{Remark.} By Proposition \ref{lipschitz_continuity_solution}, for every $T > 0$, the It\^o map $x(.)$ is Lipschitz continuous from
\begin{displaymath}
]0,\infty[\times C^{\alpha}([0,T],\mathbb R)
\textrm{ into }
C^0([0,T],]0,\infty[),
\end{displaymath}
where $C^{\alpha}([0,T],\mathbb R)$ is equipped with $\|.\|_{\infty,T}$ or $\|.\|_{\alpha,T}$.
%

% Proposition : Continuous differentiability of the It map.

%
\begin{proposition}\label{differentiability_Ito_map}
Under Assumption \ref{drift_assumption}, the It\^o map $x(.)$ is continuously differentiable from
\begin{displaymath}
]0,\infty[\times C^{\alpha}([0,T],\mathbb R)
\textrm{ into }
C^0([0,T],]0,\infty[)
\end{displaymath}
for every $T > 0$.
\end{proposition}
%

% Proof.

%
\begin{proof}
Consider $(x_0,w)\in\textrm E := ]0,\infty[\times C^{\alpha}([0,T],\mathbb R)$ for $T > 0$ arbitrarily chosen,
\begin{displaymath}
m_0\in\left]
0,\min_{t\in [0,T]}x_t(x_0,w)\right[
\textrm{ and }
\varepsilon_0 :=
-m_0 +\min_{t\in [0,T]}x_t(x_0,w).
\end{displaymath}
Since $x(.)$ is continuous from $\textrm E$ into $C^0([0,T],]0,\infty[)$ by Proposition \ref{lipschitz_continuity_solution} :
\begin{eqnarray}
 \forall\varepsilon\in ]0,\varepsilon_0]
 \textrm{, }
 \exists\eta & > & 0 :
 \forall (\xi,h)\in\textrm E
 \textrm{, }
 \nonumber\\
 (\xi,h) & \in &
 B_{\textrm E}((x_0,w),\eta)
 \Longrightarrow
 \|x(\xi,h) -
 x(x_0,w)\|_{\infty,T} <
 \varepsilon\leqslant\varepsilon_0.
\end{eqnarray}
In particular, for every $(\xi,h)\in B_{\textrm E}((x_0,w),\eta)$, the function $x(\xi,h)$ is $[m_0,M_0]$-valued with $[m_0,M_0]\subset ]0,\infty[$ and
\begin{displaymath}
M_0 :=
-m_0 +\min_{t\in [0,T]}x_t(x_0,w) +
\max_{t\in [0,T]}x_t(x_0,w).
\end{displaymath}
Then, since the function $b$ is $[1/\alpha] + 1$ times continuously differentiable on $]0,\infty[$ and has bounded derivatives on $[m_0,M_0]$ by Assumption \ref{drift_assumption}.(1) ; $x(.)$ is continuously differentiable from $B_{\textrm E}((x_0,w),\eta)$ into $C^0([0,T],]0,\infty[)$ by Friz and Victoir \cite{FV10}, theorems 11.3 and 11.6.
\\
\\
That finishes the proof, because $(x_0,w)$ has been arbitrarily chosen.
\end{proof}
\noindent
\textbf{Remarks :}
\begin{enumerate}
 \item In order to derive the It\^o map with respect to the driving signal at point $w$ in the direction $h\in C^{\beta}([0,T],\mathbb R^d)$, $\beta\in ]0,1[$ has to satisfy the condition $\alpha +\beta > 1$ to ensure the existence of the geometric $1/\alpha$-rough path over $w +\varepsilon h$ ($\varepsilon > 0$) provided at Friz and Victoir \cite{FV10}, Theorem 9.34 when $d > 1$. That condition can be avoided when $d = 1$, because the canonical geometric $1/\alpha$-rough path over $w +\varepsilon h$ is
 \begin{displaymath}
 t\in [0,T]
 \longmapsto
 \left(1,w_t +\varepsilon h_t,\dots,
 \frac{(w_t +\varepsilon h_t)^{[1/\alpha]}}{[1/\alpha]!}\right).
 \end{displaymath}
 \item The first order directional derivative of $x(.)$ at point $(x_0,w)\in\textrm E$ in the direction $(\xi,h)\in\textrm E$ is denoted by $\textrm D_{(\xi,h)}x_.(x_0,w)$ and
 \begin{displaymath}
 \textrm D_{(\xi,h)}x_t(x_0,w) =
 \xi +
 \int_{0}^{t}\dot b[x_s(x_0,w)]
 \textrm D_{(\xi,h)}x_s(x_0,w)ds +
 \sigma h_t
 \end{displaymath}
 for every $t\in [0,T]$. Then,
 \begin{displaymath}
 \textrm D_{(\xi,h)}x_.(x_0,w) =
 \int_{0}^{.}(\xi +\sigma h_s)\exp\left[\int_{s}^{.}\dot b[x_u(x_0,w)]du\right]ds.
 \end{displaymath}
 So, by Assumption \ref{drift_assumption}.(2) :
 \begin{displaymath}
 |\textrm D_{(\xi,h)}x_t(x_0,w)|\leqslant
 T(\xi +\sigma\|h\|_{\infty,T})
 \end{displaymath}
 for every $t\in [0,T]$.
\end{enumerate}
The end of the subsection is devoted to three consequences of propositions \ref{lipschitz_continuity_solution} and \ref{differentiability_Ito_map} on the partial It\^o map $x(.,w)$ for $w\in C^{\alpha}(\mathbb R_+,\mathbb R)$ arbitrarily fixed.
%

% Corollary : Monotonicity of $x(.,w)$.

%
\begin{corollary}\label{monotonicity_Ito_map}
Under Assumption \ref{drift_assumption}, $x_t(.,w)$ is (strictly) increasing on $]0,\infty[$ for every $t > 0$.
\end{corollary}
%

% Proof.

%
\begin{proof}
By Proposition \ref{differentiability_Ito_map} :
\begin{eqnarray*}
 \frac{\partial}{\partial x_0}x_t(x_0,w) & = &
 \textrm D_{(1,0)}x_t(x_0,w)\\
 & = &
 \int_{0}^{t}\exp\left[\int_{s}^{t}\dot b[x_u(x_0,w)]du\right]ds > 0
\end{eqnarray*}
for every $t > 0$. That finishes the proof.
\end{proof}
%

% Corollary : Extension of $x(.,w)$ at point 0.

%
\begin{corollary}\label{extension_Ito_map}
Under Assumption \ref{drift_assumption}, there exists $x(0,w)\in C^{\alpha}(\mathbb R_+,\mathbb R_+)$ such that $x_t(0,w) > 0$ for every $t > 0$, and
\begin{displaymath}
\lim_{x_0\rightarrow 0}
\|x(x_0,w) - x(0,w)\|_{\infty,T} = 0
\textrm{ $;$ }
\forall T > 0.
\end{displaymath}
\end{corollary}
%

% Proof.

%
\begin{proof}
The existence of the limit $x(0,w)$ of $x(.,w)$ in $C^0(\mathbb R_+,\mathbb R_+)$ when the initial condition $x_0$ goes down to 0 is proved in a first step. At the second step, it is shown that $x_t(0,w) > 0$ for every $t > 0$.
\\
\\
\textbf{Step 1.} Consider a strictly positive real sequence $(x_{0}^{n})_{n\in\mathbb N}$ such that :
\begin{displaymath}
\lim_{n\rightarrow\infty}
x_{0}^{n} = 0.
\end{displaymath}
Let $T > 0$ be arbitrarily chosen. By Proposition \ref{lipschitz_continuity_solution} :
\begin{displaymath}
\|x(x_{0}^{n},w) - x(x_{0}^{m},w)\|_{\infty,T}
\leqslant
|x_{0}^{n} - x_{0}^{m}|
\textrm{ ; }
\forall m,n\in\mathbb N.
\end{displaymath}
Then, since $C^0([0,T],\mathbb R)$ is a Banach space, $x(x_{0}^{n},w)|_{[0,T]}$ converges in $C^0([0,T],\mathbb R_+)$ when $n$ goes to infinity. Since the strictly positive real sequence $(x_{0}^{n})_{n\in\mathbb N}$ has been arbitrarily chosen, there exists a function $x(0,w|_{[0,T]})$ belonging to $C^0([0,T],\mathbb R_+)$ such that :
\begin{displaymath}
\lim_{x_0\rightarrow 0}
\|x(x_0,w) - x(0,w|_{[0,T]})\|_{\infty,T} = 0.
\end{displaymath}
Consider the function $x(0,w) :\mathbb R_+\rightarrow\mathbb R_+$ such that :
\begin{displaymath}
x(0,w)|_{[0,T]} :=
x(0,w|_{[0,T]})
\end{displaymath}
for every $T > 0$. By construction, $x(0,w)$ is the limit of $x(.,w)$ in $C^0(\mathbb R_+,\mathbb R_+)$ when the initial condition $x_0$ goes down to 0.
\\
\\
\textbf{Step 2.} For $t > s\geqslant 0$ and $x_0 > 0$ arbitrarily chosen :
\begin{eqnarray*}
 x_t(x_0,w) - x_s(x_0,w) -\sigma(w_t - w_s) & = &
 \int_{s}^{t}b[x_u(x_0,w)]du\\
 & \geqslant &
 (t - s)b\left[\sup_{u\in [s,t]}x_u(x_0,w)\right]
\end{eqnarray*}
by Assumption \ref{drift_assumption}.(2). Assume that $x_u(0,w) = 0$ for every $u\in [s,t]$. Then,
\begin{eqnarray*}
 \lim_{x_0\rightarrow 0}
 x_t(x_0,w) - x_s(x_0,w) -\sigma(w_t - w_s) & \geqslant &
 \lim_{x_0\rightarrow 0}
 (t - s)b\left[\sup_{u\in [s,t]}x_u(x_0,w)\right]\\
 & = & \infty
\end{eqnarray*}
by Assumption \ref{drift_assumption}.(4). However, by the first step of the proof :
\begin{eqnarray*}
 \lim_{x_0\rightarrow 0}
 x_t(x_0,w) - x_s(x_0,w) -\sigma(w_t - w_s) & = &
 x_t(0,w) - x_s(0,w) -\sigma(w_t - w_s)\\
 & < & \infty.
\end{eqnarray*}
Therefore, for every $s > t\geqslant 0$, there exists $u\in [s,t]$ such that $x_u(0,w) > 0$.
\\
\\
In particular, there exists a strictly positive real sequence $(t_n)_{n\in\mathbb N}$ such that $t_n\downarrow 0$ when $n\rightarrow\infty$, and
\begin{displaymath}
x_{t_n}(0,w) > 0
\textrm{ $;$ }
\forall n\in\mathbb N.
\end{displaymath}
Let $n\in\mathbb N$ be arbitrarily chosen. Since $x(0,w)$ is continuous on $\mathbb R_+$ by construction, $x_t(0,w) > 0$ for every $t\in [t_n,\tau_0(t_n)[$, where
\begin{displaymath}
\tau_0(t_n) :=
\inf\{t > t_n : x_t(0,w) = 0\}.
\end{displaymath}
For every $t\in [0,\tau_0(t_n) - t_n[$, consider
\begin{displaymath}
\tau_{\min}(n,t) :=
\textrm{argmin}_{s\in [t_n,t_n + t]}
x_s(0,w).
\end{displaymath}
Let $t\in [0,\tau_0(t_n) - t_n[$ be arbitrarily chosen. By Assumption \ref{drift_assumption}.(2) and Corollary \ref{monotonicity_Ito_map} :
\begin{displaymath}
b[x_s(x_0,w)]\leqslant
b[x_s(0,w)]\leqslant
b[x_{\tau_{\min}(n,t)}(0,w)] <\infty
\end{displaymath}
for every $s\in [t_n,t_n + t]$ and $x_0 > 0$. Then, by Lebesgue's theorem :
\begin{eqnarray*}
 x_{t_n + t}(0,w) & = &
 x_{t_n}(0,w) +
 \lim_{x_0\rightarrow 0}
 \int_{t_n}^{t_n + t}b[x_s(x_0,w)]ds +
 \sigma(w_{t_n + t} - w_{t_n})\\
 & = &
 x_{t_n}(0,w) +
 \int_{0}^{t}b[x_{t_n + s}(0,w)]ds +
 \sigma w_{t}^{t_n}
\end{eqnarray*}
with $w^{t_n} := w_{t_n + .} - w_{t_n}$ on $\mathbb R_+$. Therefore, $x_{t_n + .}(0,w) = x[x_{t_n}(0,w),w^{t_n}]$ on $[0,\tau_0(t_n)- t_n[$. Since $x_{t_n}(0,w) > 0$ and $w^{t_n}$ belongs to $C^{\alpha}(\mathbb R_+,\mathbb R)$, by Proposition \ref{existence_solution} :
\begin{eqnarray*}
 \tau_0(t_n) & = & \inf\{t > 0 :
 x_{t_n + t}(0,w) = 0\}\\
 & = & \inf\{t > 0 :
 x_t[x_{t_n}(0,w),w^{t_n}] = 0\}\\
 & = & \infty.
\end{eqnarray*}
So, $x(0,w)$ is a $]0,\infty[$-valued function on $[t_n,\infty[$ for every $n\in\mathbb N$. Since $t_n\downarrow 0$ when $n\rightarrow\infty$, $x(0,w)$ is a $]0,\infty[$-valued function on $]0,\infty[$.
\end{proof}
%

% Corollary : Lipschitz continuity of $x(.,w)$.

%
\begin{corollary}\label{Lipschitz_continuity_IC}
Under Assumption \ref{drift_assumption} :
\begin{displaymath}
|x_t(x_{0}^{1},w) - x_t(x_{0}^{2},w)|
\leqslant
|x_{0}^{1} - x_{0}^{2}|e^{-Kt}
\end{displaymath}
for every $x_{0}^{1},x_{0}^{2},t\in\mathbb R_+$.
\end{corollary}
%

% Proof.

%
\begin{proof}
Put $x^1 := x(x_{0}^{1},w)$ and $x^2 := x(x_{0}^{2},w)$ for $x_{0}^{1},x_{0}^{2} > 0$ such that $x_{0}^{1}\not= x_{0}^{2}$. By Proposition \ref{monotonicity_Ito_map}, $x_{t}^{1}\not= x_{t}^{2}$ for every $t\in\mathbb R_+$.
\\
\\
The function $x^1 - x^2$ satisfies :
\begin{equation}\label{equation_difference_IC}
x_{t}^{1} - x_{t}^{2} =
x_{0}^{1} - x_{0}^{2} +
\int_{0}^{t}[b(x_{s}^{1}) - b(x_{s}^{2})]ds
\textrm{ ; }
\forall t\in\mathbb R_+.
\end{equation}
Let $t\in\mathbb R_+$ be arbitrarily chosen. By Equation (\ref{equation_difference_IC}) :
\begin{eqnarray*}
 (x_{t}^{1} - x_{t}^{2})^2 & = &
 (x_{0}^{1} - x_{0}^{2})^2 +
 2\int_{0}^{t}(x_{s}^{1} - x_{s}^{2})d(x^1 - x^2)_s\\
 & = &
 (x_{0}^{1} - x_{0}^{2})^2 +
 2\int_{0}^{t}(x_{s}^{1} - x_{s}^{2})[b(x_{s}^{1}) - b(x_{s}^{2})]ds.
\end{eqnarray*}
Then,
\begin{equation}\label{Lipschitz_continuity_IC_equality}
\frac{\partial}{\partial t}(x_{t}^{1} - x_{t}^{2})^2 =
2(x_{t}^{1} - x_{t}^{2})^2\frac{b(x_{t}^{1}) - b(x_{t}^{2})}{x_{t}^{1} - x_{t}^{2}}.
\end{equation}
By Assumption \ref{drift_assumption}.(2) :
\begin{displaymath}
\forall u > 0\textrm{, }
\dot b(u) < -K.
\end{displaymath}
Then, by the mean-value theorem, there exists $c_t\in ]x_{t}^{1}\wedge x_{t}^{2},x_{t}^{1}\vee x_{t}^{2}[$ such that :
\begin{displaymath}
\frac{b(x_{t}^{1}) - b(x_{t}^{2})}{x_{t}^{1} - x_{t}^{2}} =
\dot b(c_t) < -K.
\end{displaymath}
Therefore, by Equation (\ref{Lipschitz_continuity_IC_equality}) :
\begin{displaymath}
\frac{\partial}{\partial t}(x_{t}^{1} - x_{t}^{2})^2
\leqslant -2K(x_{t}^{1} - x_{t}^{2})^2.
\end{displaymath}
In conclusion,
\begin{equation}\label{Lipschitz_continuity_IC_conclusion}
|x_{t}^{1} - x_{t}^{2}|
\leqslant
|x_{0}^{1} - x_{0}^{2}|e^{-Kt}.
\end{equation}
If $x_{0}^{1} = 0$, $x_{0}^{2} = 0$ or $x_{0}^{1} = x_{0}^{2}$, Inequality (\ref{Lipschitz_continuity_IC_conclusion}) holds true.
\end{proof}
%

% Subsection : Existence, uniqueness and convergence of the implicit Euler scheme.

%
\subsection{Existence, uniqueness and convergence of the implicit Euler scheme}
Let $T > 0$ and $n\in\mathbb N^*$ be arbitrarily fixed, and consider a dissection $(t_{0}^{n},t_{1}^{n}\dots,t_{n}^{n})$ of $[0,T]$.
\\
\\
The subsection deals with the global existence, the uniqueness, an estimate and the convergence of the implicit Euler scheme associated to Equation (\ref{main_equation_deterministic}) and to the dissection $(t_{0}^{n},t_{1}^{n},\dots,t_{n}^{n})$ :
\begin{equation}\label{Euler_scheme}
x_{k + 1}^{n} =
x_{k}^{n} + b(x_{k + 1}^{n})(t_{k + 1}^{n} - t_{k}^{n}) +
\sigma(w_{t_{k + 1}^{n}} - w_{t_{k}^{n}})
\end{equation}
with $x_{0}^{n} := x_0 > 0$.
%

% Proposition : Global existence of the implicit Euler scheme.

%
\begin{proposition}\label{existence_Euler_scheme}
Under Assumption \ref{drift_assumption}, Equation (\ref{Euler_scheme}) has a unique $]0,\infty[$-valued solution on $\{0,\dots,n\}$.
\end{proposition}
%

% Proof.

%
\begin{proof}
Let $\lambda > 0$ and $\mu\in\mathbb R$ be arbitrarily chosen, and put $\varphi(x) := \mu + \lambda b(x) - x$ for every $x > 0$.
\\
\\
By Assumption \ref{drift_assumption}.(1)-(2), the function $\varphi$ is continuously differentiable on $]0,\infty[$, and
\begin{displaymath}
\dot\varphi(x) =
\lambda\dot b(x) - 1 < 0
\end{displaymath}
for every $x > 0$. So, $\varphi$ is strictly decreasing on $]0,\infty[$. By Assumption \ref{drift_assumption}.(4) :
\begin{displaymath}
\lim_{x\rightarrow 0^+}
\varphi(x) =
\mu +\lambda\lim_{x\rightarrow 0^+}b(x) =
\infty.
\end{displaymath}
Let $x > x_* > 0$ be arbitrarily chosen. By Assumption \ref{drift_assumption}.(2) :
\begin{displaymath}
b(x) < -K(x - x_*) + b(x_*).
\end{displaymath}
Then,
\begin{displaymath}
\varphi(x) < -(\lambda K + 1)x +\mu +\lambda[Kx_* + b(x_*)].
\end{displaymath}
So,
\begin{displaymath}
\lim_{x\rightarrow\infty}
b(x) = -\infty.
\end{displaymath}
Therefore, the equation $\varphi(x) = 0$ has a unique solution belonging to $]0,\infty[$.
\\
\\
In conclusion, by recurrence, Equation (\ref{Euler_scheme}) has a unique $]0,\infty[$-valued solution on $\{0,\dots,n\}$.
\end{proof}
%

% Proposition : Estimate of the Euler scheme.

%
\begin{proposition}\label{estimate_Euler_scheme}
Under Assumption \ref{drift_assumption}, the solution $x^n$ of Equation (\ref{Euler_scheme}) \mbox{satisfies :}
\begin{displaymath}
\max_{k\in\{0,\dots,n\}}
x_{k}^{n}
\leqslant
x_0 + |b(x_0)|T + 2\sigma\|w\|_{\infty,T}.
\end{displaymath}
\end{proposition}
%

% Proof.

%
\begin{proof}
Let $k\in\{1,\dots,n\}$ be arbitrarily chosen, and put
\begin{displaymath}
n(x_0,k) :=\max\{i\in\{0,\dots,k\} : x_{i}^{n}\leqslant x_0\}.
\end{displaymath}
If $n(x_0,k) = k$, then $0 < x_{k}^{n}\leqslant x_0$. Assume that $n(x_0,k) < k$. Then,
\begin{eqnarray*}
 x_{k}^{n} - x_{n(x_0,k)}^{n} & = &
 \sum_{i = n(x_0,k)}^{k - 1}
 x_{i + 1}^{n} - x_{i}^{n}\\
 & = &
 \sigma[w_{t_{k}^{n}} - w_{t_{n(x_0,k)}^{n}}] +
 \sum_{i = n(x_0,k)}^{k - 1}
 b(x_{i + 1}^{n})(t_{i + 1}^{n} - t_{i}^{n}).
\end{eqnarray*}
By Assumption \ref{drift_assumption}.(2) :
\begin{eqnarray*}
 \sum_{i = n(x_0,k)}^{k - 1}
 b(x_{i + 1}^{n})(t_{i + 1}^{n} - t_{i}^{n}) & \leqslant &
 b(x_0)[t_{k}^{n} - t_{n(x_0,k)}^{n}]\\
 & \leqslant &
 |b(x_0)|T.
\end{eqnarray*}
Therefore,
\begin{displaymath}
0 < x_{k}^{n}\leqslant
x_0 + |b(x_0)|T + 2\sigma\|w\|_{\infty,T}.
\end{displaymath}
That finishes the proof.
\end{proof}
\noindent
\textbf{Notations.} Throughout the subsection, the solution of Equation (\ref{main_equation_deterministic}) is denoted by $x$ instead of $x(x_0,w)$ for the sake of readability. The solution of Equation (\ref{Euler_scheme}) is denoted by $x^n$. For every $t\in ]0,T]$, put
\begin{displaymath}
x_{t}^{n} :=
\sum_{k = 0}^{n - 1}
\left[x_{k}^{n} +
\frac{x_{k + 1}^{n} - x_{k}^{n}}{t_{k + 1}^{n} - t_{k}^{n}}(t - t_{k}^{n})\right]\mathbf 1_{]t_{k}^{n},t_{k + 1}^{n}]}(t).
\end{displaymath}
The function $t\in [0,T]\mapsto x_{t}^{n}$ is also denoted by $x^n$ and called the step-$n$ implicit Euler scheme associated to Equation (\ref{main_equation_deterministic}) and to the dissection $(t_{0}^{n},t_{1}^{n},\dots,t_{n}^{n})$.
\\
\\
In the sequel, $t_{k}^{n} := kT/n$ for every $n\in\mathbb N^*$ and $k\in\{0,\dots,n\}$.
%

% Theorem : Convergence of the Euler scheme.

%
\begin{theorem}\label{convergence_Euler_scheme}
Under Assumption \ref{drift_assumption} :
\begin{eqnarray*}
 \|x^n - x\|_{\infty,T}
 & \leqslant &
 [(\|\dot b\|_{\infty,[x_*,x^*]}^{2} +\|\dot b\|_{\infty,[x_*,x^*]} + 1)\|x\|_{\alpha,T} + \|b\|_{\infty,[x_*,x^*]} + \|w\|_{\alpha,T}]\times\\
 & &
 (T^{\alpha}\vee T^{\alpha + 2})n^{-\alpha}
\end{eqnarray*}
with
\begin{displaymath}
x_* :=
\inf_{t\in [0,T]}
x_t
\textrm{ and }
x^* :=
\sup_{t\in [0,T]}
x_t.
\end{displaymath}
\end{theorem}
%

% Proof.

%
\begin{proof}
Consider the vector $(\xi_{0}^{n},\dots,\xi_{n}^{n})$ defined by $\xi_{k}^{n} := x_{t_{k}^{n}}$ for every $k\in\{0,\dots,n\}$. By Equation (\ref{main_equation_deterministic}) :
\begin{displaymath}
\xi_{k + 1}^{n} =
\xi_{k}^{n} +
b(\xi_{k + 1}^{n})(t_{k + 1}^{n} - t_{k}^{n}) +\sigma(w_{t_{k + 1}^{n}} - w_{t_{k}^{n}}) +
\varepsilon_{k}^{n}
\end{displaymath}
with
\begin{displaymath}
\varepsilon_{k}^{n} :=
-\int_{t_{k}^{n}}^{t_{k + 1}^{n}}
[b(\xi_{k + 1}^{n}) - b(x_t)]dt
\end{displaymath}
for every $k\in\{0,\dots,n - 1\}$.
\\
\\
Let $k\in\{1,\dots,n\}$ and $i\in\{0,\dots,k - 1\}$ be arbitrarily chosen. If $x_{i + 1}^{n} >\xi_{i + 1}^{n}$, since $b$ is strictly decreasing on $]0,\infty[$ by Assumption \ref{drift_assumption}.(2) :
\begin{displaymath}
b(x_{i + 1}^{n}) - b(\xi_{i + 1}^{n})\leqslant 0.
\end{displaymath}
Then,
\begin{eqnarray}
 |x_{i + 1}^{n} -\xi_{i + 1}^{n}| & = &
 x_{i + 1}^{n} -\xi_{i + 1}^{n}
 \nonumber\\
 & = &
 x_{i}^{n} -\xi_{i}^{n} +
 [b(x_{i + 1}^{n}) - b(\xi_{i + 1}^{n})](t_{i + 1}^{n} - t_{i}^{n}) -\varepsilon_{i}^{n}
 \nonumber\\
 \label{estimate_Euler_scheme_1}
 & \leqslant &
 |x_{i}^{n} -\xi_{i}^{n}| +
 |\varepsilon_{i}^{n}|.
\end{eqnarray}
If $x_{i + 1}^{n}\leqslant\xi_{i + 1}^{n}$, since $b$ is strictly decreasing on $]0,\infty[$ by Assumption \ref{drift_assumption}.(2) :
\begin{displaymath}
b(\xi_{i + 1}^{n}) - b(x_{i + 1}^{n})\leqslant 0.
\end{displaymath}
Then,
\begin{eqnarray}
 |x_{i + 1}^{n} -\xi_{i + 1}^{n}| & = &
 \xi_{i + 1}^{n} - x_{i + 1}^{n}
 \nonumber\\
 & = &
 \xi_{i}^{n} - x_{i}^{n} +
 [b(\xi_{i + 1}^{n}) - b(x_{i + 1}^{n})](t_{i + 1}^{n} - t_{i}^{n}) +\varepsilon_{i}^{n}
 \nonumber\\
 \label{estimate_Euler_scheme_2}
 & \leqslant &
 |x_{i}^{n} -\xi_{i}^{n}| +
 |\varepsilon_{i}^{n}|.
\end{eqnarray}
So, by inequalities (\ref{estimate_Euler_scheme_1}) and (\ref{estimate_Euler_scheme_2}) together :
\begin{displaymath}
|x_{i + 1}^{n} -\xi_{i + 1}^{n}|
\leqslant
|x_{i}^{n} -\xi_{i}^{n}| +
|\varepsilon_{i}^{n}|.
\end{displaymath}
By recurrence :
\begin{equation}\label{estimate_Euler_scheme_3}
|x_{k}^{n} -\xi_{k}^{n}|
\leqslant
\sum_{i = 0}^{k - 1}|\varepsilon_{i}^{n}|.
\end{equation}
By Assumption \ref{drift_assumption}.(1), $b$ is Lipschitz continuous on $[x_*,x^*]$. Then,
\begin{eqnarray*}
 |\varepsilon_{i}^{n}|
 & \leqslant &
 \|\dot b\|_{\infty,[x_*,x^*]}
 \|x\|_{\alpha,T}
 \int_{t_{i}^{n}}^{t_{i + 1}^{n}}
 (t_{i + 1}^{n} - t)^{\alpha}dt\\
 & \leqslant &
 \|\dot b\|_{\infty,[x_*,x^*]}
 \|x\|_{\alpha,T}
 \frac{T^{\alpha + 1}}{n^{\alpha + 1}}.
\end{eqnarray*}
So, by Equation (\ref{estimate_Euler_scheme_3}) :
\begin{equation}\label{estimate_Euler_scheme_4}
|x_{k}^{n} -\xi_{k}^{n}|
\leqslant
\|\dot b\|_{\infty,[x_*,x^*]}
\|x\|_{\alpha,T}
\frac{T^{\alpha + 1}}{n^{\alpha}}.
\end{equation}
Let $t\in ]0,T]$ be arbitrarily chosen. There exists $k\in\{0,\dots,n - 1\}$ such that $t\in ]t_{k}^{n}, t_{k + 1}^{n}]$. By Inequality (\ref{estimate_Euler_scheme_4}) :
\begin{eqnarray}
 |x_{k + 1}^{n} - x_{k}^{n}| & \leqslant &
 [|[b(x_{k + 1}^{n}) - b(\xi_{k + 1}^{n})| + |b(\xi_{k + 1}^{n})|](t_{k + 1}^{n} - t_{k}^{n}) +
 \|w\|_{\alpha,T}(t_{k + 1}^{n} - t_{k}^{n})^{\alpha}
 \nonumber\\
 & \leqslant &
 [[\|\dot b\|_{\infty,[x_*,x^*]}|x_{k + 1}^{n} -\xi_{k + 1}^{n}| + \|b\|_{\infty,[x_*,x^*]}]T + \|w\|_{\alpha,T}T^{\alpha}]n^{-\alpha}
 \nonumber\\
 \label{estimate_Euler_scheme_5}
 & \leqslant &
 [\|\dot b\|_{\infty,[x_*,x^*]}^{2}\|x\|_{\alpha,T} + \|b\|_{\infty,[x_*,x^*]} + \|w\|_{\alpha,T}](T^{\alpha}\vee T^{\alpha + 2})n^{-\alpha}.
\end{eqnarray}
By inequalities (\ref{estimate_Euler_scheme_4}) and (\ref{estimate_Euler_scheme_5}) together :
\begin{eqnarray*}
 |x_{t}^{n} - x_t| & \leqslant &
 |x_{t}^{n} - x_{k}^{n}| +
 |x_{k}^{n} -\xi_{k}^{n}| +
 |\xi_{k}^{n} - x_t|\\
 & \leqslant &
 |x_{k + 1}^{n} - x_{k}^{n}| +
 (\|\dot b\|_{\infty,[x_*,x^*]} + 1)
 \|x\|_{\alpha,T}
 (T^{\alpha}\vee T^{\alpha + 1})n^{-\alpha}\\
 & \leqslant &
 [(\|\dot b\|_{\infty,[x_*,x^*]}^{2} +\|\dot b\|_{\infty,[x_*,x^*]} + 1)\|x\|_{\alpha,T} + \|b\|_{\infty,[x_*,x^*]} + \|w\|_{\alpha,T}]\times\\
 & &
 (T^{\alpha}\vee T^{\alpha + 2})n^{-\alpha}.
\end{eqnarray*}
That finishes the proof.
\end{proof}
%

% Section : Probabilistic and statistical properties of the solution.

%
\section{Probabilistic and statistical properties of the solution}
\noindent
Let $(\Omega,\mathcal A,\mathbb P)$ be the canonical probability space associated to the stochastic process $B$.
\\
\\
The solution of Equation (\ref{main_equation}) is the stochastic process $X(x_0) := (X_t(x_0))_{t\in\mathbb R_+}$ such that :
\begin{displaymath}
X_t(x_0,\omega) :=
x_t[x_0,B(\omega)]
\end{displaymath}
for every $\omega\in\Omega$ and $t\in\mathbb R_+$.
\\
\\
\textbf{Notations :}
\begin{itemize}
 \item The expectation operator associated to the probability measure $\mathbb P$ is denoted by $\mathbb E$.
 \item For every $p > 0$, the space of random variables $U :\Omega\rightarrow\mathbb R$ such that $\mathbb E(|U|^p) <\infty$ is denoted by $L^p(\Omega,\mathbb P)$ and equipped with its usual \mbox{norm $\|.\|_p$.}
\end{itemize}
\noindent
Under Assumption \ref{drift_assumption}, if $B$ is a centered Gaussian process with locally $\alpha$-H\"older continuous paths, by Proposition \ref{solution_estimate} together with Fernique's theorem (see Fernique \cite{FERNIQUE70}) :
\begin{displaymath}
\|X(x_0)\|_{\infty,T}\in L^p(\Omega,\mathbb P)
\end{displaymath}
for every $p,T > 0$.
\\
\\
The section deals with probabilistic and statistical properties of $X(x_0)$, obtained via its deterministic properties proved previously and various additional conditions on the signal $B$.
%

% Subsection : Ergodicity of the solution.

%
\subsection{Ergodicity of the solution}
Assume that $B$ is a two-sided fractional Brownian motion of Hurst parameter $H\in ]0,1[$ ($\alpha\in ]0,H[$).
\\
\\
Let $\theta := (\theta_t)_{t\in\mathbb R}$ be the dynamical system on $(\Omega,\mathcal A)$, called Wiener shift, such that :
\begin{displaymath}
\theta_t\omega :=\omega_{t + .} -\omega_t
\end{displaymath}
for every $\omega\in\Omega$ and $t\in\mathbb R$. By Maslowski and Schmalfuss \cite{MS04}, $(\Omega,\mathcal A,\mathbb P,\theta)$ is an metric dynamical system (i.e.
\begin{itemize}
 \item $(t,\omega)\in\mathbb R\times\Omega\longmapsto\theta_t\omega$ is $\mathcal B(\mathbb R)\otimes\mathcal A,\mathcal A$-measurable.
 \item For every $t\in\mathbb R$, $\theta_t\mathbb P =\mathbb P$ where
 \begin{displaymath}
 (\theta_t\mathbb P)(A) :=
 \mathbb P(\{\omega\in\Omega :\theta_t\omega\in A\})
 \textrm{ $;$ }
 \forall A\in\mathcal A),
 \end{displaymath}
\end{itemize}
which is ergodic.
%

% Lemma : Control of the fractional Brownian motion.

%
\begin{lemma}\label{fBm_control}
There exists a $\theta$-invariant set $\Omega^*\in\mathcal A$ satisfying $\mathbb P(\Omega^*) = 1$, such that for every $\omega\in\Omega^*$,
\begin{displaymath}
\exists C(\omega) > 0 :
\forall t\in\mathbb R\textrm{$,$ }
|B_t(\omega)|\leqslant
C(\omega)(1 + |t|^2).
\end{displaymath}
\end{lemma}
\noindent
For a proof, see Gess et al. \cite{GLR11}, Lemma 3.3 generalizing Maslowski and Schmalfuss \cite{MS04}, Lemma 2.6.
\\
\\
\textbf{Remarks :}
\begin{enumerate}
 \item In the sequel, $\Omega^*$ is equipped with the trace $\sigma$-algebra
 \begin{displaymath}
 \mathcal A^* :=
 \{A\cap\Omega^*
 \textrm{ $;$ }A\in\mathcal A\}.
 \end{displaymath}
 \item $(\Omega^*,\mathcal A^*,\mathbb P,\theta)$ is also an ergodic metric dynamical system.
\end{enumerate}
The map
\begin{displaymath}
X(.) : (\omega,x_0,t)\in\Omega\times\mathbb R_{+}^{2}
\longmapsto
X_t(x_0,\omega)
\end{displaymath}
is a continuous random dynamical system on $(\mathbb R_+,\mathcal B(\mathbb R_+))$ over the metric dynamical systems $(\Omega,\mathcal A,\mathbb P,\theta)$ and $(\Omega^*,\mathcal A^*,\mathbb P,\theta)$.
\\
\\
The reader can refer to Arnold \cite{ARNOLD98} on random dynamical systems.
\\
\\
\textbf{Notation.} Let $(W_t)_{t\in\mathbb R_+}$ be a stochastic process on $(\Omega,\mathcal A,\mathbb P)$. For every $\omega\in\Omega$ and $t,T\in\mathbb R_+$,
\begin{displaymath}
W_{t,T}(\omega) := W_t(\theta_{-T}\omega).
\end{displaymath}
%

% Proposition : Control of the solution.

%
\begin{proposition}\label{epsilon_estimate}
Under Assumption \ref{drift_assumption}, for every $\omega\in\Omega^*$, there exists a constant $C(\omega) > 0$ such that for every $t,T,x_0\in\mathbb R_+$ and $\varepsilon\geqslant x_0$,
\begin{displaymath}
|X_{t,T}(x_0,\omega) -\varepsilon|
\leqslant
\varepsilon +
|b(\varepsilon)|t +
C(\omega)(1 + t + T)^2.
\end{displaymath}
\end{proposition}
%

% Proof.

%
\begin{proof}
Let $\omega\in\Omega^*$, $t,T\in\mathbb R_+$ and $\varepsilon\geqslant x_0 > 0$ be arbitrarily chosen, and put
\begin{displaymath}
\tau_{t}^{-}(\varepsilon,\theta_{-T}\omega) :=
\sup\{s\in [0,t] : X_s(x_0,\theta_{-T}\omega)\leqslant\varepsilon\}.
\end{displaymath}
If $\tau_{t}^{-}(\varepsilon,\theta_{-T}\omega) = t$, then
\begin{displaymath}
|X_{t,T}(x_0,\omega) - \varepsilon|\leqslant\varepsilon.
\end{displaymath}
Assume that $\tau_{t}^{-}(\varepsilon,\theta_{-T}\omega) < t$. Then,
\begin{eqnarray}
 X_{t,T}(x_0,\omega) & = &
 X_{\tau_{t}^{-}(\varepsilon,\theta_{-T}\omega),T}(x_0,\omega) +
 \int_{\tau_{t}^{-}(\varepsilon,\theta_{-T}\omega)}^{t}
 b[X_{s,T}(x_0,\omega)]ds +
 \nonumber\\
 & &
 \sigma[B_{t,T}(\omega) - B_{\tau_{t}^{-}(\varepsilon,\theta_{-T}\omega),T}(\omega)]
 \nonumber\\
 \label{epsilon_estimate_1}
 & = &
 \varepsilon +
 \int_{\tau_{t}^{-}(\varepsilon,\theta_{-T}\omega)}^{t}
 b[X_{s,T}(x_0,\omega)]ds +
 \sigma[B_{t - T}(\omega) - B_{\tau_{t}^{-}(\varepsilon,\theta_{-T}\omega) - T}(\omega)].
\end{eqnarray}
On one hand, by Assumption \ref{drift_assumption}.(2) :
\begin{eqnarray}
 \int_{\tau_{t}^{-}(\varepsilon,\theta_{-T}\omega)}^{t}
 b[X_{s,T}(x_0,\omega)]ds
 & \leqslant &
 b(\varepsilon)[t -\tau_{t}^{-}(\varepsilon,\theta_{-T}\omega)]
 \nonumber\\
 \label{epsilon_estimate_2}
 & \leqslant &
 |b(\varepsilon)|t.
\end{eqnarray}
On the other hand, by Lemma \ref{fBm_control}, there exists a constant $C_0(\omega) > 0$, not depending on $t$, $T$, $x_0$ and $\varepsilon$, such that :
\begin{eqnarray}
 |B_{t - T}(\omega) - B_{\tau_{t}^{-}(\varepsilon,\theta_{-T}\omega) - T}(\omega)|
 & \leqslant &
 C_0(\omega)[2 + |t - T|^2 + |\tau_{t}^{-}(\varepsilon,\theta_{-T}\omega) - T|^2]
 \nonumber\\
 \label{epsilon_estimate_3}
 & \leqslant &
 4C_0(\omega)(1 + t + T)^2.
\end{eqnarray}
Therefore, by Equality (\ref{epsilon_estimate_1}) together with inequalities (\ref{epsilon_estimate_2}) and (\ref{epsilon_estimate_3}) :
\begin{displaymath}
0\leqslant
X_{t,T}(x_0,\omega) -\varepsilon
\leqslant
|b(\varepsilon)|t +
4\sigma C_0(\omega)(1 + t + T)^2.
\end{displaymath}
In conclusion, by putting $C(\omega) := 4\sigma C_0(\omega)$, for every $t,T\in\mathbb R_+$,
\begin{equation}\label{epsilon_estimate_final}
|X_{t,T}(x_0,\omega) -\varepsilon|
\leqslant
\varepsilon +
|b(\varepsilon)|t +
C(\omega)(1 + t + T)^2.
\end{equation}
If $x_0 = 0$, Inequality (\ref{epsilon_estimate_final}) holds true.
\end{proof}
%

% Theorem : Existence of a stationary solution.

%
\begin{theorem}\label{stationary_solution_CIR}
Under Assumption \ref{drift_assumption}, there exists a random variable $X^* :\Omega\rightarrow\mathbb R_+$ belonging to $L^p(\Omega,\mathbb P)$ for every $p > 0$, such that for every $x_0\in\mathbb R_+$,
\begin{displaymath}
\left|X_T(x_0) - X^*\circ\theta_T\right|
\xrightarrow[T\rightarrow\infty]{} 0
\end{displaymath}
almost surely and in $L^p(\Omega,\mathbb P)$ for every $p > 0$.
\end{theorem}
%

% Proof.

%
\begin{proof}
Let $\omega\in\Omega^*$, $t,x_0\in\mathbb R_+$, $n\in\mathbb N$ and $p > 0$ be arbitrarily chosen.
\\
\\
\textbf{Almost sure convergence.} By the cocycle property of the random dynamical system $X(.)$, Corollary \ref{Lipschitz_continuity_IC} and Proposition \ref{epsilon_estimate} ; there exists a constant $C(\omega) > 0$, not depending on $t$, $n$ and $x_0$, such that for every $\varepsilon\geqslant x_0$,
\begin{small}
\begin{eqnarray}
 |X_n(x_0,\theta_{-n}\omega) -
 X_{n + 1}(x_0,\theta_{-(n + 1)}\omega)|
 & = &
 |X_n(x_0,\theta_{-n}\omega) -
 X_n[X_1(x_0,\theta_{-(n + 1)}\omega),\theta_{-n}\omega]|
 \nonumber\\
 & \leqslant &
 e^{-Kn}|x_0 - X_1(x_0,\theta_{-(n + 1)}\omega)|
 \nonumber\\
 \label{ergodicity_CIR_estimate_1}
 & \leqslant &
 e^{-Kn}[|x_0 -\varepsilon| +
 |X_1(x_0,\theta_{-(n + 1)}\omega) -\varepsilon|]\\
 & \leqslant &
 e^{-Kn}[|x_0 -\varepsilon| +
 \varepsilon + |b(\varepsilon)| + C(\omega)(3 + n)^2].
 \nonumber
\end{eqnarray}
\end{small}
\newline
Since $n^k =_{n\rightarrow\infty}\textrm o(e^{Kn})$ for every $k\in\mathbb N$, $(X_n(x_0,\theta_{-n}\omega))_{n\in\mathbb N}$ is a Cauchy sequence, and its limit $X^0(\omega)$ is not depending on $x_0$ because for every other initial condition $x_1 > 0$,
\begin{displaymath}
|X_n(x_0,\theta_{-n}\omega) -
X_n(x_1,\theta_{-n}\omega)|
\leqslant
e^{-Kn}|x_0 - x_1|
\xrightarrow[n\rightarrow\infty]{} 0.
\end{displaymath}
For every $\varepsilon\geqslant x_0$,
\begin{small}
\begin{eqnarray}
 |X_t(x_0,\theta_{-t}\omega) - X^0(\omega)|
 & \leqslant &
 |X_t(x_0,\theta_{-t}\omega) - X_{[t]}(x_0,\theta_{-[t]}\omega)| +
 |X_{[t]}(x_0,\theta_{-[t]}\omega) - X^0(\omega)|
 \nonumber\\
 & = &
 |X_{[t]}[X_{t - [t]}(x_0,\theta_{-t}\omega),\theta_{-[t]}\omega] -
 X_{[t]}(x_0,\theta_{-[t]}\omega)| +
 \nonumber\\
 & &
 |X_{[t]}(x_0,\theta_{-[t]}\omega) - X^0(\omega)|
 \nonumber\\
 \label{ergodicity_CIR_estimate_2}
 & \leqslant &
 e^{-K[t]}[
 |x_0 -\varepsilon| +
 |X_{t - [t]}(x_0,\theta_{-t}\omega) -\varepsilon|] +\\
 & &
 |X_{[t]}(x_0,\theta_{-[t]}\omega) - X^0(\omega)|
 \nonumber\\
 & \leqslant &
 e^{-K[t]}[|x_0 -\varepsilon| +\varepsilon + |b(\varepsilon)| + C(\omega)(2 + [t])^2] +
 \nonumber\\
 & &
 |X_{[t]}(x_0,\theta_{-[t]}\omega) - X^0(\omega)|.
 \nonumber
\end{eqnarray}
\end{small}
\newline
Therefore,
\begin{equation}\label{ergodicity_CIR_1}
\lim_{t\rightarrow\infty}
|X_t(x_0,\theta_{-t}\omega) - X^0(\omega)| = 0
\end{equation}
because $[t]^k =_{t\rightarrow\infty}\textrm o(e^{K[t]})$ for every $k\in\mathbb N$. By the cocycle property of the random dynamical system $X(.)$ :
\begin{eqnarray}
 \label{ergodicity_CIR_2}
 X_t[X_n(x_0,\theta_{-n}\omega),\omega] & = &
 X_{t + n}(x_0,\theta_{-n}\omega)\\
 & = &
 X_{t + n}[x_0,\theta_{-(t + n)}(\theta_t\omega)].
 \nonumber
\end{eqnarray}
By continuity of $X(.,\omega)$ from $\mathbb R_+$ into $C^0(\mathbb R_+)$, Corollary \ref{extension_Ito_map} and (\ref{ergodicity_CIR_1}) ; when $n$ goes to infinity in Equality (\ref{ergodicity_CIR_2}) :
\begin{displaymath}
X_t[X^0(\omega),\omega] = X^0(\theta_t\omega).
\end{displaymath}
Since $(\Omega^*,\mathcal A^*,\mathbb P,\theta)$ is an ergodic metric dynamical system and $X^0$ is a (generalized) random fixed point of the continuous random dynamical system $X(.)$, $(X^0\circ\theta_t)_{t\in\mathbb R_+}$ is a stationary solution of Equation (\ref{main_equation}). Therefore, for every $\omega\in\Omega^*$,
\begin{displaymath}
\lim_{t\rightarrow\infty}
|X_t(x_0,\omega) - X^0(\theta_t\omega)| = 0
\end{displaymath}
because all solutions of Equation (\ref{main_equation}) converge pathwise forward to each other in time by Corollary \ref{Lipschitz_continuity_IC}.
\\
\\
\textbf{Convergence in $L^p(\Omega,\mathbb P)$.} Since $B$ and $(B_{s - t} - B_{-t})_{s\in\mathbb R}$ have the same distribution $\mathbb P$, for every $U\in L^p(\Omega,\mathbb P)$ and $s\in\mathbb R_+$,
\begin{equation}\label{equality_expectation}
\|X_s(x_0)\circ\theta_{-t} - U\|_p =
\|X_s(x_0) - U\circ\theta_t\|_p.
\end{equation}
By Inequality (\ref{ergodicity_CIR_estimate_1}) and Equality (\ref{equality_expectation}), for every $\varepsilon\geqslant x_0$,
\begin{displaymath}
\|X_n(x_0)\circ\theta_{-n} -
X_{n + 1}(x_0)\circ\theta_{-(n + 1)}\|_p
\leqslant e^{-Kn}[|x_0 -\varepsilon| +
\|X_1(x_0,\omega) -\varepsilon\|_p].
\end{displaymath}
Then, since the set $L^p(\Omega,\mathbb P)$ equipped with $\|.\|_p$ is a Banach space, there exists $X^*\in L^p(\Omega,\mathbb P)$ such that :
\begin{displaymath}
\lim_{n\rightarrow\infty}
\|X_n(x_0)\circ\theta_{-n} - X^*\|_p = 0
\end{displaymath}
and $X^*(\omega) = X^0(\omega)$ for every $\omega\in\Omega^*$. By Inequality (\ref{ergodicity_CIR_estimate_2}) and Equality (\ref{equality_expectation}), for every $\varepsilon\geqslant x_0$,
\begin{eqnarray*}
 \|X_t(x_0)\circ\theta_{-t} - X^*\|_p
 & \leqslant &
 e^{-K[t]}\left[|x_0 -\varepsilon| +
 \sup_{s\in [0,1]}
 \|X_s(x_0) -\varepsilon\|_p\right] +\\
 & &
 \|X_{[t]}(x_0)\circ\theta_{-[t]} - X^*\|_p.
\end{eqnarray*}
Then,
\begin{displaymath}
\lim_{t\rightarrow\infty}
\|X_t(x_0)\circ\theta_{-t} - X^*\|_p= 0.
\end{displaymath}
Therefore, by Equality (\ref{equality_expectation}) :
\begin{eqnarray*}
 \lim_{t\rightarrow\infty}
 \|X_t(x_0) - X^*\circ\theta_t\|_p & = &
 \lim_{t\rightarrow\infty}
 \|X_t(x_0)\circ\theta_{-t} - X^*\|_p\\
 & = & 0.
\end{eqnarray*}
\end{proof}
%

% Corollary : Ergodic theorem.

%
\begin{corollary}\label{ergodic_theorem_CIR}
Under Assumption \ref{drift_assumption}, for every uniformly continuous function $\varphi :\mathbb R_+\rightarrow\mathbb R$ with polynomial growth, and every $x_0\in\mathbb R_+$,
\begin{displaymath}
\frac{1}{T}\int_{0}^{T}\varphi[X_t(x_0)]dt
\xrightarrow[T\rightarrow\infty]{}
\mathbb E[\varphi(X^*)]
\end{displaymath}
almost surely and in $L^p(\Omega,\mathbb P)$ for every $p > 0$.
\end{corollary}
%

% Proof.

%
\begin{proof}
Let $\omega\in\Omega^*$, $x_0\in\mathbb R_+$ and $p > 0$ be arbitrarily chosen. Consider also
\begin{displaymath}
I_T(\varphi,x_0) :=
\frac{1}{T}\int_{0}^{T}\varphi[X_t(x_0)]dt
\textrm{ ; }
\forall T > 0
\end{displaymath}
where $\varphi :\mathbb R_+\rightarrow\mathbb R$ is a uniformly continuous function such that :
\begin{equation}\label{ergodic_theorem_CIR_phi}
\forall x\in\mathbb R_+
\textrm{, }
|\varphi(x)|\leqslant
c(1 + x^n)
\end{equation}
with $c > 0$ and $n\in\mathbb N^*$.
\\
\\
\textbf{Almost sure convergence.} On one hand, since $\varphi$ has a polynomial growth and $X^*$ belongs to $L^p(\Omega,\mathbb P)$ for every $p > 0$ by Theorem \ref{stationary_solution_CIR}, $\varphi(X^*)$ too. Moreover, $(\Omega^*,\mathcal A^*,\mathbb P,\theta)$ is an ergodic metric dynamical system, then by Birkhoff's theorem :
\begin{equation}\label{ergodic_theorem_CIR_1}
\lim_{T\rightarrow\infty}
\frac{1}{T}\int_{0}^{T}\varphi[X^*(\theta_t\omega)]dt =
\mathbb E[\varphi(X^*)].
\end{equation}
On the other hand, by Theorem \ref{stationary_solution_CIR} together with the uniform continuity of $\varphi$, for every $\varepsilon > 0$, there exists $T_0 > 0$ such that :
\begin{displaymath}
\forall t > T_0\textrm{, }
|\varphi[X_t(x_0,\omega)] -\varphi[X^*(\theta_t\omega)]|
\leqslant\frac{\varepsilon}{2}.
\end{displaymath}
Then, for every $T > T_0$,
\begin{displaymath}
\frac{1}{T}
\int_{T_0}^{T}|\varphi[X_t(x_0,\omega)] -\varphi[X^*(\theta_t\omega)]|dt
\leqslant
\frac{\varepsilon}{2}.
\end{displaymath}
\newline
Moreover, there exists $T_1 > T_0$ such that for every $T > T_1$,
\begin{displaymath}
\frac{1}{T}
\int_{0}^{T_0}|\varphi[X_t(x_0,\omega)] -\varphi[X^*(\theta_t\omega)]|dt
\leqslant\frac{\varepsilon}{2}.
\end{displaymath}
So,
\begin{displaymath}
\frac{1}{T}\left|\int_{0}^{T}
[\varphi[X_t(x_0,\omega)] -\varphi[X^*(\theta_t\omega)]]dt\right|
\leqslant
\varepsilon.
\end{displaymath}
Therefore, by definition :
\begin{equation}\label{ergodic_theorem_CIR_2}
\lim_{T\rightarrow\infty}
\frac{1}{T}\int_{0}^{T}
[\varphi[X_t(x_0,\omega)] -\varphi[X^*(\theta_t\omega)]]dt = 0.
\end{equation}
By (\ref{ergodic_theorem_CIR_1}) and (\ref{ergodic_theorem_CIR_2}) together :
\begin{displaymath}
\lim_{T\rightarrow\infty} I_T(\varphi,x_0,\omega) =
\mathbb E[\varphi(X^*)].
\end{displaymath}
\textbf{Convergence in $L^p(\Omega,\mathbb P)$.} For every $t\in\mathbb R_+$ and $q > 0$,
\begin{eqnarray*}
 \|X_t(x_0)\|_q
 & \leqslant &
 \|X_t(x_0) - X^*\circ\theta_t\|_q +
 \|X^*\circ\theta_t\|_q\\
 & = &
 \|X_t(x_0) - X^*\circ\theta_t\|_q +
 \|X^*\|_q.
\end{eqnarray*}
Then, since $\|X_t(x_0) - X^*\circ\theta_t\|_q\rightarrow 0$ when $t$ goes to infinity by Theorem \ref{stationary_solution_CIR} :
\begin{displaymath}
\sup_{t\in\mathbb R_+}
\|X_t(x_0)\|_q < \infty
\textrm{ $;$ }\forall q > 0.
\end{displaymath}
Therefore, by (\ref{ergodic_theorem_CIR_phi}) :
\begin{eqnarray*}
 \sup_{T > 0}\|I_T(\varphi,x_0)\|_p
 & \leqslant &
 \sup_{t\in\mathbb R_+}
 \|\varphi[X_t(x_0)]\|_p\\
 & \leqslant &
 c\left[1 +\sup_{t\in\mathbb R_+}
 \mathbb E^{1/p}[X_{t}^{np}(x_0)]\right] <\infty.
\end{eqnarray*}
In conclusion, by Vitali's theorem :
\begin{displaymath}
\lim_{T\rightarrow\infty}
\|I_T(\varphi,x_0) -\mathbb E[\varphi(X^*)]\|_p = 0.
\end{displaymath}
\end{proof}
%

% Proposition : Hitting times of $x_b$.

%
\begin{proposition}\label{hitting_times_x_b}
Under Assumption \ref{drift_assumption}, the equation $b(x) = 0$ has a unique solution $x_b > 0$ such that for every $t_*\in\mathbb R_+$,
\begin{displaymath}
\tau_{t_*}(x_b) :=
\inf\{t > t_* : X_t(x_0) = x_b\} <\infty
\end{displaymath}
almost surely.
\end{proposition}
%

% Proof.

%
\begin{proof}
By Assumption \ref{drift_assumption}.(1)-(2), the equation $b(x) = 0$ has a unique solution $x_b > 0$ such that $b(x) > 0$ (resp. $b(x) < 0$) for every $x\in ]0,x_b[$ (resp. $x > x_b$). Let $t_*\in\mathbb R_+$ be arbitrarily chosen, and consider $\omega\in\{\tau_{t_*}(x_b) =\infty\}$.
\\
\\
Without loss of generality, assume that $\sigma > 0$.
\\
\\
On one hand, assume that $X_{t_*}(x_0,\omega)\geqslant x_b$. Since $X(x_0,\omega)$ is continuous, $X_s(x_0,\omega) > x_b$ for every $s > t_*$. So, for every $t > t_*$,
\begin{eqnarray*}
 x_b & < &
 X_t(x_0,\omega)\\
 & = &
 X_{t_*}(x_0,\omega) +
 \int_{t_*}^{t}b[X_s(x_0,\omega)]ds +
 \sigma[B_t(\omega) - B_{t_*}(\omega)]\\
 & \leqslant &
 X_{t_*}(x_0,\omega) +\sigma[B_t(\omega) - B_{t_*}(\omega)].
\end{eqnarray*}
However, by Molchan \cite{MOLCHAN02} :
\begin{displaymath}
\inf\{t > 0 : B_t(\omega) = \lambda\} <\infty
\end{displaymath}
for every $\lambda < B_{t_*}(\omega) + 1/\sigma[x_b - X_{t_*}(x_0,\omega)]$. There is a contradiction.
\\
\\
On the other hand, assume that $X_{t_*}(x_0,\omega) < x_b$. Since $X(x_0,\omega)$ is continuous, $X_s(x_0,\omega) < x_b$ for every $s > t_*$. So, for every $t > t_*$,
\begin{eqnarray*}
 x_b & > &
 X_t(x_0,\omega)\\
 & = &
 X_{t_*}(x_0,\omega) +
 \int_{t_*}^{t}b[X_s(x_0,\omega)]ds +
 \sigma[B_t(\omega) - B_{t_*}(\omega)]\\
 & \geqslant &
 X_{t_*}(x_0,\omega) +\sigma[B_t(\omega) - B_{t_*}(\omega)].
\end{eqnarray*}
However, by Molchan \cite{MOLCHAN02} :
\begin{displaymath}
\inf\{t > 0 : B_t(\omega) = \lambda\} <\infty
\end{displaymath}
for every $\lambda > B_{t_*}(\omega) + 1/\sigma[x_b - X_{t_*}(x_0,\omega)]$. There is a contradiction.
\\
\\
Therefore, $\mathbb P[\tau_{t_*}(x_b) =\infty] = 0$. That finishes the proof.
\end{proof}
\noindent
\textbf{Remark.} By Proposition \ref{hitting_times_x_b}, at any time $t_*\in\mathbb R_+$, the stochastic process $X(x_0)$ will hit again $x_b$ on $]t_*,\infty[$. In particular, for almost every $\omega\in\Omega$, if $X(x_0,\omega)$ has a limit when $t$ goes to infinity, it cannot be different from $x_b$.
%

% Subsection : Absolute continuity of the distribution of the solution.

%
\subsection{Absolute continuity of the distribution of the solution}
Let $T > 0$ be arbitrarily fixed, and assume that $B$ is a centered Gaussian process defined on $[0,T]$, with $\alpha$-H\"older continuous paths.
\\
\\
The subsection deals with applications of the Malliavin calculus to the absolute continuity of the distribution of $X_t(x_0)$ for every $t\in ]0,T]$. The reader can refer to Nualart \cite{NUALART06} on Malliavin calculus.
\\
\\
Let $\mathcal H$ be the reproducing kernel Hilbert space of $B$, and consider an orthonormal basis $(h_n)_{n\in\mathbb N}$ of $\mathcal H$. The Wiener integral with respect to $B$, defined on $\mathcal H$, is denoted by $\mathbf B$. The Malliavin derivative associated to $(\Omega,\mathcal A,\mathbb P)$, $\mathcal H$ and $\mathbf B$ is denoted by $\mathbf D$.
\\
\\
Let $\mathcal H^1$ be the Cameron-Martin space of $B$. The map $\mathbf I :\mathcal H\rightarrow\mathcal H^1$ defined by
\begin{displaymath}
\mathbf I_.(h) :=
\mathbb E[\mathbf B(h)B_.]
\textrm{ ; }
\forall h\in\mathcal H,
\end{displaymath}
is an isometry between $\mathcal H$ and $\mathcal H^1$ (see Marie \cite{MARIE_sensitivities}, Lemma 3.4).
\\
\\
\textbf{Notations :}
\begin{itemize}
 \item The domain of the Malliavin derivative is denoted by $\mathbb D^{1,2}$.
 \item Consider a random variable $U :\Omega\rightarrow\mathbb R$ and a normed vector space $\textrm E$ continuously embedded in $\Omega$ ($\textrm E\hookrightarrow\Omega$). For every $\omega\in\Omega$ and $e\in\textrm E$, $U^{\omega}(e) := U(\omega + e)$.
\end{itemize}
Until the end of the subsection, $B$ satisfies the following assumption.
%

% Assumption : Gaussian assumption.

%
\begin{assumption}\label{Gaussian_assumption}
$B$ is a centered Gaussian process defined on $[0,T]$, with $\alpha$-H\"older continuous paths, such that :
\begin{enumerate}
 \item The covariance function $\normalfont\textrm R$ of $B$ satisfies $\normalfont\textrm R(t,t) > 0$ for every $t\in ]0,T]$.
 \item $\langle\varphi_1,\psi_1\rangle_{\mathcal H}\geqslant\langle\varphi_2,\psi_2\rangle_{\mathcal H}$ for every $\varphi_1,\varphi_2,\psi_1,\psi_2\in\mathcal H$ such that
 \begin{displaymath}
 \varphi_1(t)\geqslant\varphi_2(t)\geqslant 0\textrm{ and }
 \psi_1(t)\geqslant\psi_2(t)\geqslant 0\textrm{ $;$ }
 \forall t\in [0,T].
 \end{displaymath}
 \item The Cameron-Martin space of $B$ is continuously embedded in $C^{\alpha}([0,T],\mathbb R)$.
\end{enumerate}
\end{assumption}
\noindent
\textbf{Example.} A fractional Brownian motion of Hurst parameter $H\in ]0,1[$ satisfies Assumption \ref{Gaussian_assumption} for every $\alpha\in ]0,H[$ (See Friz and Victoir \cite{FV10}, Section 15.2.2, and Nualart \cite{NUALART06}, Section 5.1.3).
%

% Proposition : Malliavin derivative of the solution and absolute continuity of its distribution.

%
\begin{proposition}\label{Malliavin_derivative_absolute_continuity}
Under assumptions \ref{drift_assumption} and \ref{Gaussian_assumption}, $X_t(x_0)\in\mathbb D^{1,2}$ and
\begin{displaymath}
\mathbf D_sX_t(x_0) =
\sigma\mathbf 1_{[0,t]}(s)
\exp\left[\int_{s}^{t}\dot b[X_u(x_0)]du\right]
\end{displaymath}
for every $s,t\in [0,T]$. Moreover, the distribution of $X_t(x_0)$ has a density with respect to the Lebesgue measure on $(\mathbb R,\mathcal B(\mathbb R))$ for every $t\in ]0,T]$.
\end{proposition}
%

% Proof.

%
\begin{proof}
Let $\omega\in\Omega$ and $s,t\in [0,T]$ be arbitrarily chosen. Since $\mathcal H^1\hookrightarrow C^{\alpha}([0,T],\mathbb R)$ by Assumption \ref{Gaussian_assumption}.(3), by Proposition \ref{differentiability_Ito_map} :
\begin{displaymath}
h\in\mathcal H^1
\longmapsto
X_{t}^{\omega}(x_0,h)
\end{displaymath}
is continuously differentiable. In other words, $X_t(x_0)$ is continuously $\mathcal H^1$-differentiable. Then, by Nualart \cite{NUALART06}, Proposition 4.1.3, $X_t(x_0)$ is locally differentiable in the sense of Malliavin, and
\begin{displaymath}
\langle\mathbf DX_t(x_0)(\omega),h\rangle_{\mathcal H} =
\textrm D_{\mathbf I(h)}X_{t}^{\omega}(x_0,0)
\end{displaymath}
for every $h\in\mathcal H$. By Proposition \ref{differentiability_Ito_map} :
\begin{small}
\begin{eqnarray*}
 \mathbf D_sX_t(x_0)(\omega) & = &
 \sum_{n = 0}^{\infty}
 h_n(s)\langle\mathbf DX_t(x_0)(\omega),h_n\rangle_{\mathcal H}\\
 & = &
 \sum_{n = 0}^{\infty}
 h_n(s)\textrm D_{\mathbf I(h_n)}X_{t}^{\omega}(x_0,0)\\
 & = &
 \sum_{n = 0}^{\infty}
 h_n(s)\left[
 \int_{0}^{t}\dot b[X_u(x_0,\omega)]
 \textrm D_{\mathbf I(h_n)}X_{u}^{\omega}(x_0,0)du +
 \sigma\mathbf I_t(h_n)\right]\\
 & = &
 \int_{0}^{t}
 \dot b[X_u(x_0,\omega)]
 \sum_{n = 0}^{\infty}
 h_n(s)\textrm D_{\mathbf I(h_n)}X_{u}^{\omega}(x_0,0)du +
 \sigma
 \sum_{n = 0}^{\infty}
 h_n(s)\textrm D_{\mathbf I(h_n)}B_{t}^{\omega}(0)\\
 & = &
 \int_{0}^{t}
 \dot b[X_u(x_0,\omega)]
 \textbf D_sX_u(x_0)(\omega)du +
 \sigma\mathbf D_sB_t(\omega).
\end{eqnarray*}
\end{small}
\newline
Since $\mathbf DB_t =\mathbf 1_{[0,t]}$, $\mathbf D_sX_.(x_0)(\omega)$ satisfies
\begin{displaymath}
\mathbf D_sX_t(x_0)(\omega) = \xi +
\int_{0}^{t}
\dot b[X_u(x_0,\omega)]
\mathbf D_sX_u(x_0)(\omega)du
\end{displaymath}
with $\xi = 0$ (resp. $\xi =\sigma$) for $t\in [0,s[$ (resp. $t\in [s,T]$). Then,
\begin{displaymath}
\mathbf D_sX_t(x_0) =
\sigma\mathbf 1_{[0,t]}(s)
\exp\left[\int_{s}^{t}\dot b[X_u(x_0)]du\right].
\end{displaymath}
So, by Assumption \ref{drift_assumption}.(2) :
\begin{equation}\label{estimate_Malliavin_1}
\sigma\mathbf 1_{[0,t]}(s)
\exp\left[\int_{0}^{T}\dot b[X_u(x_0)]du\right]
\leqslant
\mathbf D_sX_t(x_0)
\leqslant
\sigma\mathbf 1_{[0,t]}(s).
\end{equation}
Put $\Gamma_t :=\|\mathbf DX_t(x_0)\|_{\mathcal H}^{2}$. By Assumption \ref{Gaussian_assumption}.(1)-(2) together with Inequality (\ref{estimate_Malliavin_1}) :
\begin{equation}\label{estimate_Malliavin_2}
0 <\sigma^2\textrm R(t,t)
\exp\left[2\int_{0}^{T}\dot b[X_u(x_0)]du\right]
\leqslant\Gamma_t
\leqslant\sigma^2\textrm R(t,t).
\end{equation}
On one hand, by Inequality (\ref{estimate_Malliavin_2}), $\Gamma_t\in L^p(\Omega,\mathbb P)$ for every $p > 0$. So, $X_t(x_0)\in\mathbb D^{1,2}$ by Nualart \cite{NUALART06}, Lemma 4.1.2. On the other hand, by Inequality (\ref{estimate_Malliavin_2}), $\Gamma_t > 0$. So, the distribution of $X_t(x_0)$ has a density with respect to the Lebesgue measure on $(\mathbb R,\mathcal B(\mathbb R))$ by Bouleau-Hirsch's criterion (see Nualart \cite{NUALART06}, Theorem 2.1.3).
\end{proof}
\noindent
\textbf{Notation.} The Ornstein-Uhlenbeck semigroup (resp. operator) is denoted by $\mathbf T := (\mathbf T_t)_{t\in\mathbb R_+}$ (resp. $\mathbf L$). See Nualart \cite{NUALART06}, Section 1.4.
\\
\\
\textbf{Remarks :}
\begin{enumerate}
 \item Let $t\in\mathbb R_+$ be arbitrarily chosen. By Nualart \cite{NUALART06}, Property (i) page 55 :
 \begin{equation}\label{monotonicity_OU_semigroup}
 \forall U\in L^2(\Omega,\mathbb P)
 \textrm{, }
 U\geqslant 0
 \Longrightarrow
 T_t(U)\geqslant 0.
 \end{equation}
 Since $T_t$ is a linear map, (\ref{monotonicity_OU_semigroup}) implies that :
 \begin{displaymath}
 \forall U_1,U_2\in L^2(\Omega,\mathbb P)
 \textrm{, }
 U_1\geqslant U_2
 \Longrightarrow
 T_t(U_1)\geqslant T_t(U_2).
 \end{displaymath}
 \item Let $t\in ]0,T]$ be arbitrarily chosen. Bouleau-Hirsch's criterion is sufficient to prove the absolute continuity of the distribution of $X_t(x_0)$ with respect to the Lebesgue measure on $(\mathbb R,\mathcal B(\mathbb R))$, but not to provide an explicit density. Via the main result of Nourdin and Viens \cite{NV09}, the following proposition provides a density with a suitable expression.
\end{enumerate}
%

% Proposition : Explicit density.

%
\begin{proposition}\label{explicit_density}
Under assumptions \ref{drift_assumption} and \ref{Gaussian_assumption}, for every $t\in ]0,T]$,
\begin{displaymath}
\mathbb P_{X_t(x_0)}(dx) =
\frac{\mathbb E[|X_{t}^{*}(x_0)|]}{2g_{X_t(x_0)}(x)}
\exp\left[-\int_{\mathbb E[X_t(x_0)]}^{x}\frac{y -\mathbb E[X_t(x_0)]}{g_{X_t(x_0)}(y)}dy\right]dx
\end{displaymath}
where $X_{t}^{*}(x_0) := X_t(x_0) -\mathbb E[X_t(x_0)]$ and
\begin{displaymath}
g_{X_t(x_0)}(x) :=
\mathbb E[\langle\mathbf DX_t(x_0),-\mathbf D\mathbf L^{-1}X_t(x_0)\rangle_{\mathcal H}|
X_t(x_0) = x]
\end{displaymath}
for every $x > 0$.
\end{proposition}
%

% Proof.

%
\begin{proof}
Let $t\in ]0,T]$ and $s\in [0,T]$ be arbitrarily chosen. By Nourdin and Viens \cite{NV09}, Proposition 3.7, Inequality (\ref{estimate_Malliavin_1}) and (\ref{monotonicity_OU_semigroup}) :
\begin{eqnarray*}
 -\mathbf D_s\mathbf L^{-1}X_t(x_0) & = &
 \int_{0}^{\infty}
 e^{-u}\mathbf T_u[\mathbf D_sX_t(x_0)]du\\
 & \geqslant &
 \sigma\mathbf 1_{[0,t]}(s)
 \int_{0}^{\infty}e^{-u}\mathbf T_u\left[\exp\left[\int_{0}^{T}\dot b[X_v(x_0)]dv\right]\right]du.
\end{eqnarray*}
Then, by Assumption \ref{Gaussian_assumption}.(1)-(2) together with Inequality (\ref{estimate_Malliavin_1}) :
\begin{eqnarray*}
 \langle\mathbf DX_t(x_0),-\mathbf D\mathbf L^{-1}X_t(x_0)\rangle_{\mathcal H}
 & \geqslant &
 \sigma^2\textrm R(t,t)
 \exp\left[\int_{0}^{T}\dot b[X_v(x_0)]dv\right]\times\\
 & &
 \int_{0}^{\infty}e^{-u}\mathbf T_u\left[\exp\left[\int_{0}^{T}\dot b[X_v(x_0)]dv\right]\right]du > 0.
\end{eqnarray*}
So,
\begin{eqnarray*}
 g_{X_{t}^{*}(x_0)}[X_{t}^{*}(x_0)] & := &
 \mathbb E[\langle\mathbf DX_{t}^{*}(x_0),-\mathbf D\mathbf L^{-1}X_{t}^{*}(x_0)\rangle_{\mathcal H}|
 X_{t}^{*}(x_0)]\\
 & = &
 \mathbb E[\langle\mathbf DX_t(x_0),-\mathbf D\mathbf L^{-1}X_t(x_0)\rangle_{\mathcal H}|
 X_{t}^{*}(x_0)] > 0.
\end{eqnarray*}
Therefore, by Nourdin and Viens \cite{NV09}, Theorem 3.1 :
\begin{displaymath}
\mathbb P_{X_{t}^{*}(x_0)}(dx) =
\frac{\mathbb E[|X_{t}^{*}(x_0)|]}{2g_{X_{t}^{*}(x_0)}(x)}
\exp\left[-\int_{0}^{x}\frac{y}{g_{X_{t}^{*}(x_0)}(y)}dy\right]dx.
\end{displaymath}
Together with a straightforward application of the transfer theorem, that finishes the proof.
\end{proof}
%

% Subsection : Integrability and convergence of the implicit Euler scheme.

%
\subsection{Integrability and convergence of the implicit Euler scheme}
Let $T > 0$ be arbitrarily fixed, and assume that $B$ is a centered Gaussian process defined on $[0,T]$, with $\alpha$-H\"older continuous paths.
\\
\\
Let $n\in\mathbb N^*$ be arbitrarily chosen, and consider the dissection $(t_{0}^{n},t_{1}^{n},\dots,t_{n}^{n})$ of $[0,T]$ such that $t_{k}^{n} := kT/n$ for every $k\in\{0,\dots,n\}$. Consider also the stochastic process $X^n(x_0) := (X_{t}^{n}(x_0))_{t\in [0,T]}$ such that for every $\omega\in\Omega$, $X^n(x_0,\omega)$ is the step-$n$ implicit Euler scheme associated to Equation (\ref{main_equation_deterministic}) driven by $B(\omega)$ and to the dissection $(t_{0}^{n},t_{1}^{n},\dots,t_{n}^{n})$.
%

% Proposition : Probabilistic convergence of the Euler scheme.

%
\begin{proposition}\label{probabilistic_convergence_Euler}
Under Assumption \ref{drift_assumption} :
\begin{displaymath}
\|X^n(x_0) - X(x_0)\|_{\infty,T}
\xrightarrow[n\rightarrow\infty]{\textrm{a.s.}} 0
\end{displaymath}
with rate of convergence $\normalfont{\textrm O}(n^{-\alpha})$. Moreover, for every $p > 0$,
\begin{displaymath}
\sup_{n\in\mathbb N^*}\|X^n(x_0)\|_{\infty,T}\in
L^p(\Omega,\mathbb P)
\end{displaymath}
and
\begin{displaymath}
\lim_{n\rightarrow\infty}\mathbb E[
\|X^n(x_0) - X(x_0)\|_{\infty,T}^{p}] = 0.
\end{displaymath}
\end{proposition}
%

% Proof.

%
\begin{proof}
By Theorem \ref{convergence_Euler_scheme}, for every $\omega\in\Omega$,
\begin{displaymath}
\|X^n(x_0,\omega) - X(x_0,\omega)\|_{\infty,T}
\xrightarrow[n\rightarrow\infty]{\textrm{a.s.}} 0
\end{displaymath}
with rate of convergence $\normalfont{\textrm O}(n^{-\alpha})$.
\\
\\
Let $p > 0$ be arbitrarily chosen. By Proposition \ref{estimate_Euler_scheme} together with Fernique's theorem :
\begin{displaymath}
\sup_{n\in\mathbb N^*}\|X^n(x_0)\|_{\infty,T}\in
L^p(\Omega,\mathbb P).
\end{displaymath}
So, by Vitali's theorem :
\begin{displaymath}
\lim_{n\rightarrow\infty}\mathbb E[
\|X^n(x_0) - X(x_0)\|_{\infty,T}^{p}] = 0.
\end{displaymath}
\end{proof}
%

% Subsection : Estimation of parameters.

%
\subsection{Estimation of parameters}
The subsection deals with the estimation of the Hurst parameter and of the volatility constant of Equation (\ref{main_equation}) by using a transformation of the observations of $X(x_0)$ and already known estimators of the Hurst parameter and of the volatility constant of the fractional Ornstein-Uhlenbeck process.
\\
\\
Under Assumption \ref{drift_assumption}, for every $y_0\in\mathbb R$, let $Y(y_0) := (Y_t(y_0))_{t\in\mathbb R_+}$ be the solution of the following Langevin equation :
\begin{equation}\label{Langevin_equation}
Y_t =
y_0 - R\int_{0}^{t}Y_sds +
\sigma B_t.
\end{equation}
On the fractional Ornstein-Uhlenbeck process, see Cheridito et al. \cite{CKM03}.
%

% Proposition : Relationship between $X(x_0)$ and the associated O-U process.

%
\begin{proposition}\label{relationship_OU_CIR}
Under Assumption \ref{drift_assumption}, for every $x_0 > 0$, $y_0\in\mathbb R$ and $t\in\mathbb R_+$,
\begin{displaymath}
Y_t(y_0) =
X_t(x_0) - (x_0 - y_0)e^{-Rt} -
\int_{0}^{t}e^{-R(t - s)}b^R[X_s(x_0)]ds
\end{displaymath}
where $b^R(x) := b(x) + Rx$ for every $x > 0$.
\end{proposition}
%

% Proof.

%
\begin{proof}
Let $t\in\mathbb R_+$, $x_0 > 0$ and $y_0\in\mathbb R$ be arbitrarily chosen.
\\
\\
The stochastic process $\Delta(x_0,y_0) := X(x_0) - Y(y_0)$ satisfies :
\begin{eqnarray*}
 \Delta_t(x_0,y_0) & = &
 x_0 - y_0 +\int_{0}^{t}[b[X_s(x_0)] + RY_s(y_0)]ds\\
 & = &
 x_0 - y_0 -
 R\int_{0}^{t}\Delta_s(x_0,y_0)ds +
 \int_{0}^{t}b^R[X_s(x_0)]ds.
\end{eqnarray*}
Then,
\begin{displaymath}
\Delta_t(x_0,y_0) =
(x_0 - y_0)e^{-Rt} +
\int_{0}^{t}e^{-R(t - s)}b^R[X_s(x_0)]ds.
\end{displaymath}
That finishes the proof.
\end{proof}
\noindent
Until the end of the subsection, $B$ is a fractional Brownian motion of Hurst parameter $H\in ]0,1[$ ($\alpha\in ]0,H[$). The values of all the parameters involving in the expression of the drift function $b$ are known.
\\
\\
Consider the map $\Theta$ from $C^0(\mathbb R_+,]0,\infty[)$ into $C^0(\mathbb R_+,\mathbb R)$ such that :
\begin{displaymath}
\Theta(\varphi)(t) :=
\varphi(t) -
\int_{0}^{t}e^{-R(t - s)}b^R[\varphi(s)]ds
\textrm{ ; }
\forall t\in\mathbb R_+,
\end{displaymath}
for every $\varphi\in C^0(\mathbb R_+,]0,\infty[)$. By Proposition \ref{relationship_OU_CIR} :
\begin{displaymath}
Y(x_0) =
\Theta[X(x_0)].
\end{displaymath}
Since the parameter $(H,\sigma)$ doesn't involve in the expression of the map $\Theta$, an observation $x(x_0)$ of $X(x_0)$ provides an observation of $Y(x_0)$ by applying the transformation $\Theta$ to $x(x_0)$. So, since Equation (\ref{main_equation}) has the same additive noise $\sigma B$ than Equation (\ref{Langevin_equation}), a consistent estimator of $(H,\sigma)$ for the fractional Ornstein-Uhlenbeck process $Y(x_0)$ provides an estimation of the real value of $(H,\sigma)$ from the observation $y(x_0) :=\Theta[x(x_0)]$.
\\
For $H\in ]1/2,1[$, there are several papers dealing with the estimation of the Hurst parameter and of the volatility constant of the fractional Ornstein-Uhlenbeck process. Some estimators use the whole path of $Y(x_0)$ (see Berzin and Le\'on \cite{BL08}), and some estimators use discrete observations of $Y(x_0)$ (see Melichov \cite{MELICHOV11} or Brouste and Iacus \cite{BI13}).
\\
\\
In order to get an observation of $Y(x_0)$ at the time $t\in\mathbb R_+$ from $y(x_0)$, the whole path $x(x_0)$ has to be known until the time $t$ by construction of the map $\Theta$. In practice, only discrete observations of $X(x_0)$ are available. So, with the same arguments, the end of the subsection deals with a consistent estimator of $H$ for $X^n(x_0)$ instead of $X(x_0)$.
\\
Under Assumption \ref{drift_assumption}, let $Y^n(x_0)$ be the step-$n$ implicit Euler scheme associated to Equation (\ref{Langevin_equation}) and to the dissection $(t_{0}^{n},t_{1}^{n},\dots,t_{n}^{n})$ of $[0,T]$ :
\begin{equation}\label{Langevin_implicit_Euler}
Y_{k + 1}^{n} =
Y_{k}^{n} - RY_{k + 1}^{n}(t_{k + 1}^{n} - t_{k}^{n}) +
\sigma(B_{t_{k + 1}^{n}} - B_{t_{k}^{n}})
\end{equation}
with $Y_{0}^{n}(x_0) := x_0$. On the implicit Euler schemes associated to the fractional Langevin equation, see Garrido-Atienza et al. \cite{GAKN09}.
%

% Proposition : Relationship between $X^n(x_0)$ and the Euler scheme of the associated Langevin equation.

%
\begin{proposition}\label{relationship_Euler_OU_CIR}
Under Assumption \ref{drift_assumption}, for every $k\in\{1,\dots,n\}$,
\begin{displaymath}
Y_{k}^{n}(x_0) =
X_{k}^{n}(x_0) -
Tn^{-1}
\sum_{i = 1}^{k}
(1 + RTn^{-1})^{i - 1 - k}b_R[X_{i}^{n}(x_0)].
\end{displaymath}
\end{proposition}
%

% Proof.

%
\begin{proof}
Let $k\in\{1,\dots,n\}$ be arbitrarily chosen.
\\
\\
The stochastic process $\Delta^n(x_0) := X^n(x_0) - Y^n(x_0)$ satisfies :
\begin{eqnarray*}
 \Delta_{k + 1}^{n}(x_0) & = &
 X_{k}^{n}(x_0) - Y_{k}^{n}(x_0) +
 Tn^{-1}[b[X_{k + 1}^{n}(x_0)] + RY_{k + 1}^{n}(x_0)]\\
 & = &
 \Delta_{k}^{n}(x_0) +
 Tn^{-1}[b_R[X_{k + 1}^{n}(x_0)] - R\Delta_{k + 1}^{n}(x_0)].
\end{eqnarray*}
Then,
\begin{displaymath}
\Delta_{k + 1}^{n}(x_0) =
(1 + RTn^{-1})^{-1}\Delta_{k}^{n}(x_0) +
(1 + RTn^{-1})^{-1}b_R[X_{k}^{n}(x_0)].
\end{displaymath}
So,
\begin{displaymath}
\Delta_{k}^{n}(x_0) =
Tn^{-1}
\sum_{i = 1}^{k}
(1 + RTn^{-1})^{i - 1 - k}b_R[X_{i}^{n}(x_0)].
\end{displaymath}
That finishes the proof.
\end{proof}
\noindent
Consider the map $\Theta^n$ from $]0,\infty[^{\mathbb N}$ into $\mathbb R^{\mathbb N}$ such that :
\begin{displaymath}
\Theta_{k}^{n}(u) :=
\left\{
\begin{array}{rcl}
 u_0 & \textrm{if} & k = 0\\
 \displaystyle{u_k -
 Tn^{-1}
 \sum_{i = 1}^{k}
 (1 + RTn^{-1})^{i - 1 - k}b_R(u_i)}
 & \textrm{if} &
 k\in\{1,\dots,n\}
\end{array}
\right.
\end{displaymath}
for every $u\in]0,\infty[^{\mathbb N}$. By Proposition \ref{relationship_Euler_OU_CIR} :
\begin{displaymath}
Y^n(x_0) =
\Theta^n[X^n(x_0)].
\end{displaymath}
As for the continuous time models, the transformation $\Theta^n$ provides an observation of $Y^n(x_0)$ from an observation of $X^n(x_0)$.
%

% Definition : $k$-quadratic variation.

%
\begin{definition}\label{k_quadratic_variation}
Consider a stochastic process $W := (W_t)_{t\in\mathbb R_+}$ on $(\Omega,\mathcal A,\mathbb P)$, $\normalfont{\textrm I}_n\subset\{0,\dots,n\}$, and $k\in\mathbb N$ such that $k +\max\normalfont{\textrm I}_n\leqslant n$. The $k$-quadratic variation of $W$ with respect to the dissection $(t_{0}^{n},t_{1}^{n},\dots,t_{n}^{n})$ and to the index set $\normalfont{\textrm I}_n$ is
\begin{displaymath}
\normalfont{\textrm V}_{\normalfont{\textrm I}_n,k}(W) :=
\sum_{i\in\normalfont{\textrm I}_n}
(\Delta_kW)_{i}^{2}
\end{displaymath}
with $(\Delta_kW)_i := W_{t_{i + k}^{n}} - W_{t_{i}^{n}}$ for every $i\in\normalfont{\textrm I}_n$.
\end{definition}
%

% Proposition : Convergence of the quadratic variation of the Euler scheme of the Langevin equation.

%
\begin{proposition}\label{convergence_quadratic_variation}
Consider $\normalfont{\textrm I}_n\subset\{0,\dots,n\}$, and $k\in\mathbb N$ such that $k +\max\normalfont{\textrm I}_n\leqslant n$.
\begin{displaymath}
|\normalfont{\textrm V}_{\normalfont{\textrm I}_n,k}[Y^n(x_0)] -
\normalfont{\textrm V}_{\normalfont{\textrm I}_n,k}[Y(x_0)]|
\xrightarrow[n\rightarrow\infty]{} 0
\end{displaymath}
almost surely and in $L^p(\Omega,\mathbb P)$ for every $p > 0$.
\end{proposition}
%

% Proof.

%
\begin{proof}
Let $i\in\textrm I_n$ and $p > 0$ be arbitrarily chosen.
\begin{small}
\begin{eqnarray}
 |[\Delta_kY^n(x_0)]_{i}^{2} - [\Delta_kY(x_0)]_{i}^{2}|
 & \leqslant &
 2Y^*\times
 \nonumber\\
 & &
 \sum_{j = 1}^{k}
 |Y_{i + j}^{n}(x_0) - Y_{i + j - 1}^{n}(x_0) - [Y_{i + j}(x_0) - Y_{i + j - 1}(x_0)]|
 \nonumber\\
 & \leqslant &
 2RY^*
 \sum_{j = 1}^{k}
 \left|
 Y_{i + j}^{n}(x_0)Tn^{-1} -
 \int_{t_{i + j - 1}^{n}}^{t_{i + j}^{n}}
 Y_t(x_0)dt
 \right|
 \nonumber\\
 & \leqslant &
 2RY^*
 \sum_{j = 1}^{k}
 \int_{t_{i + j - 1}^{n}}^{t_{i + j}^{n}}
 |Y_{i + j}^{n}(x_0) - Y_t(x_0)|dt
 \nonumber\\
 \label{fOU_Euler_1}
 & \leqslant &
 2RY^*\left[Tn^{-1}
 \sum_{j = 1}^{k}
 |Y_{i + j}^{n}(x_0) - Y_{t_{i + j}^{n}}(x_0)| +\right.\\
 & &
 \left.\sum_{j = 1}^{k}
 \int_{t_{i + j - 1}^{n}}^{t_{i + j}^{n}}
 |Y_{t_{i + j}^{n}}(x_0) - Y_t(x_0)|dt\right]
 \nonumber
\end{eqnarray}
\end{small}
\newline
with
\begin{eqnarray*}
 Y^* & := &
 \|Y(x_0)\|_{\infty,T} +
 \sup_{n\in\mathbb N^*}\|Y^n(x_0)\|_{\infty,T}\\
 & &
 \in L^p(\Omega,\mathbb P).
\end{eqnarray*}
On one hand, by Garrido-Atienza et al. \cite{GAKN09}, Theorem 1 ; there exists $C(\alpha,H,T)\in L^p(\Omega,\mathbb P)$ such that :
\begin{displaymath}
\|Y^n(x_0) - Y(x_0)\|_{\infty,T}
\leqslant
C(\alpha,H,T)T^{\alpha}n^{-\alpha}.
\end{displaymath}
So,
\begin{equation}\label{fOU_Euler_2}
Tn^{-1}
\sum_{j = 1}^{k}
|Y_{i + j}^{n}(x_0) - Y_{t_{i + j}^{n}}(x_0)|
\leqslant
C(\alpha,H,T)
T^{\alpha + 1}n^{-\alpha}.
\end{equation}
On the other hand, the paths of $Y(x_0)$ are $\alpha$-H\"older continuous on $[0,T]$ with $\|Y(x_0)\|_{\alpha,T}\in L^p(\Omega,\mathbb P)$. So,
\begin{eqnarray}
 \sum_{j = 1}^{k}
 \int_{t_{i + j - 1}^{n}}^{t_{i + j}^{n}}
 |Y_{t_{i + j}^{n}}(x_0) - Y_t(x_0)|dt
 & \leqslant & \|Y(x_0)\|_{\alpha,T}
 \sum_{j = 1}^{k}
 \int_{t_{i + j - 1}^{n}}^{t_{i + j}^{n}}
 (t_{i + j}^{n} - t)^{\alpha}dt
 \nonumber\\
 \label{fOU_Euler_3}
 & \leqslant &
 \|Y(x_0)\|_{\alpha,T}T^{\alpha + 1}n^{-\alpha}.
\end{eqnarray}
By Inequality (\ref{fOU_Euler_1}) together with inequalities (\ref{fOU_Euler_2}) and (\ref{fOU_Euler_3}) :
\begin{displaymath}
|\textrm V_{\textrm I_n,k}[Y^n(x_0)] -\textrm V_{\textrm I_n,k}[Y(x_0)]|
\leqslant
2RY^*[C(\alpha,H,T) +\|Y(x_0)\|_{\alpha,T}]
T^{\alpha + 1}n^{-\alpha}.
\end{displaymath}
That finishes the proof.
\end{proof}
\noindent
Consider $\textrm I_{n}^{1} := \{0,\dots,n\}$, $\textrm I_{n}^{2} := \{2i\textrm{ $;$ }i\in\{0,\dots,[n/2]\}$,
\begin{displaymath}
\widehat H_n :=
\frac{1}{2} -\frac{1}{2\log(2)}\log\left[
\frac{\textrm V_{\textrm I_{n}^{1},1}[Y(x_0)]}{\textrm V_{\textrm I_{n}^{2},2}[Y(x_0)]}\right]
\textrm{ and }
\widehat h_n :=
\frac{1}{2} -\frac{1}{2\log(2)}\log\left[
\frac{\textrm V_{\textrm I_{n}^{1},1}[Y^n(x_0)]}{\textrm V_{\textrm I_{n}^{2},2}[Y^n(x_0)]}\right].
\end{displaymath}
%

% Proposition : Consistency of the estimator of the Hurst parameter for the Euler scheme.

%
\begin{proposition}\label{H_consistency}
$\widehat h_n$ is a strongly consistent estimator of $H$.
\end{proposition}
%

% Proof.

%
\begin{proof}
By Proposition \ref{convergence_quadratic_variation} together with the uniform continuity of $\log$ on $]0,\infty[$ :
\begin{eqnarray*}
 |\widehat h_n -\widehat H_n| & \leqslant &
 \frac{1}{2\log(2)}[
 |\log[\textrm V_{\textrm I_{n}^{1},1}[Y(x_0)]] -
 \log[\textrm V_{\textrm I_{n}^{1},1}[Y^n(x_0)]]| +\\
 & &
 |\log[\textrm V_{\textrm I_{n}^{2},2}[Y(x_0)]] -
 \log[\textrm V_{\textrm I_{n}^{2},2}[Y^n(x_0)]]|]
 \xrightarrow[n\rightarrow\infty]{\textrm{a.s.}} 0.
\end{eqnarray*}
Moreover, by Melichov \cite{MELICHOV11}, Section 3.2.1 ; $\widehat H_n$ is a strongly consistent estimator of $H$. So,
\begin{eqnarray*}
 |\widehat h_n - H| & \leqslant &
 |\widehat h_n -\widehat H_n| +
 |\widehat H_n - H|\\
 & &
 \xrightarrow[n\rightarrow\infty]{\textrm{a.s.}} 0.
\end{eqnarray*}
\end{proof}
%

% Section : Application to singular equations driven by a multiplicative noise.

%
\section{Application to singular equations driven by a multiplicative noise}
\noindent
Let $F : ]0,\infty[\rightarrow\mathbb R$ be a function satisfying the following assumption.
%

% Assumption : Assumption on the change of variable.

%
\begin{assumption}\label{change_of_variable_assumption}
\white .\black
\begin{enumerate}
 \item The function $F$ is $[1/\alpha] + 2$ times continuously differentiable on $]0,\infty[$.
 \item The function $F$ is strictly monotonic on $]0,\infty[$.
\end{enumerate}
\end{assumption}
\noindent
Under Assumption \ref{drift_assumption}, Equation (\ref{main_equation}) with the initial condition $x_0 > 0$ has a unique $]0,\infty[$-valued solution $X(x_0)$ on $\mathbb R_+$ by Proposition \ref{existence_solution}. Then, under Assumption \ref{change_of_variable_assumption}, by the rough change of variable formula :
\begin{eqnarray*}
 F[X_t(x_0)] & = &
 F(x_0) +
 \int_{0}^{t}\dot F[X_s(x_0)]dX_s(x_0)\\
 & = &
 F(x_0) +
 \int_{0}^{t}\dot F[X_s(x_0)]b[X_s(x_0)]ds +
 \sigma\int_{0}^{t}\dot F[X_s(x_0)]dB_s
\end{eqnarray*}
for every $t\in\mathbb R_+$. Therefore, by putting $\textrm I := F(]0,\infty[)$, with the initial condition $z_0\in\textrm I$, the following equation has a unique $\textrm I$-valued solution $Z(z_0) := (Z_t(z_0))_{t\in\mathbb R_+}$ on $\mathbb R_+$ :
\begin{equation}\label{main_multiplicative_equation}
Z_t =
z_0 +
\int_{0}^{t}G(Z_s)H(Z_s)ds +
\sigma\int_{0}^{t}H(Z_s)dB_s
\end{equation}
with $G := b\circ F^{-1}$ and $H :=\dot F\circ F^{-1}$.
\\
\\
\textbf{Examples.} Consider $\kappa\in\mathbb R^*$ and $u,v,w,\gamma > 0$ such that $1 -\alpha <\alpha\gamma$. Put $b(x) := u(vx^{-\gamma} - wx)$ and $F_{\kappa}(x) := x^{\kappa}$ for every $x > 0$. The function $b$ (resp. $F_{\kappa}$) satisfies Assumption \ref{drift_assumption} (resp. Assumption \ref{change_of_variable_assumption}). Then, Equation (\ref{main_multiplicative_equation}) \mbox{becomes :}
\begin{displaymath}
Z_t = z_0 +
\kappa u\int_{0}^{t}[vZ_{s}^{1 - (\gamma + 1)/\kappa} - wZ_s]ds +
\kappa\sigma\int_{0}^{t}Z_{s}^{1 - 1/\kappa}dB_s.
\end{displaymath}
On one hand, assume that $\kappa =\gamma + 1$, $u = 1/(\gamma + 1)$ and $\sigma = \zeta/(\gamma + 1)$ with $\zeta\in\mathbb R^*$. Then, by putting $\beta := 1 - 1/(\gamma + 1)$, Equation (\ref{main_multiplicative_equation}) becomes
\begin{displaymath}
Z_t = y_0 +
\int_{0}^{t}(v - wZ_s)ds +
\zeta\int_{0}^{t}Z_{s}^{\beta}dB_s
\end{displaymath}
and $\beta\in ]1 -\alpha,1[$. So, in that case, Equation (\ref{main_multiplicative_equation}) is the generalized Cox-Ingersoll-Ross model partially studied in Marie \cite{MARIE14}.
\\
\\
On the other hand, assume that $\kappa = -(\gamma + 1)$, $u = 1/(\gamma + 1)$ and $\sigma = -\zeta^*/(\gamma + 1)$ with $\zeta^*\in\mathbb R^*$. Then, by putting $\beta^* := 1/(\gamma + 1)$, Equation (\ref{main_multiplicative_equation}) becomes
\begin{displaymath}
Z_t = z_0 +
\int_{0}^{t}Z_s(w - vZ_s)ds +
\zeta^*\int_{0}^{t}Z_{s}^{1 +\beta^*}dB_s
\end{displaymath}
and $\beta^*\in ]0,\alpha[$. So, in that case, Equation (\ref{main_multiplicative_equation}) is a generalized Verhulst's model, studied for $\beta^* = 0$ and a fractional Brownian signal in Huy and Nguyen \cite{HN02}.
\\
\\
The section deals with how to transfer the probabilistic and statistical properties established at Section 3 on the solution of Equation (\ref{main_equation}) to the stochastic process $Z(z_0)$. A fractional Heston model is also introduced.
%

% Proposition : Regularity of the It map associated to the equation driven by a multiplicative noise.

%
\begin{proposition}\label{differentiability_Ito_map_multiplicative}
Under assumptions \ref{drift_assumption} and \ref{change_of_variable_assumption}, the It\^o map associated to the deterministic analog of Equation (\ref{main_multiplicative_equation}) is continuously differentiable from
\begin{displaymath}
\normalfont\textrm I\times C^{\alpha}([0,T],\mathbb R)
\textrm{ into }
C^0([0,T],\normalfont\textrm I)
\end{displaymath}
for every $T > 0$.
\end{proposition}
%

% Proof.

%
\begin{proof}
Let $T > 0$ be arbitrarily chosen. By Proposition \ref{differentiability_Ito_map}, $x(.)$ is continuously differentiable from
\begin{displaymath}
]0,\infty[\times C^{\alpha}([0,T],\mathbb R)
\textrm{ into }
C^0([0,T],]0,\infty[).
\end{displaymath}
Moreover, by Assumption \ref{change_of_variable_assumption}, $F$ and $F^{-1}$ are $[1/\alpha] + 2$ times continuously differentiable on $]0,\infty[$ and $\textrm I$ respectively. So, the map
\begin{displaymath}
(z_0,w)\longmapsto
F\circ x[F^{-1}(z_0),w]
\end{displaymath}
is continuously differentiable from
\begin{displaymath}
\textrm I\times C^{\alpha}([0,T],\mathbb R)
\textrm{ into }
C^0([0,T],\textrm I).
\end{displaymath}
\end{proof}
%

% Subsection : Probabilistic and statistical properties of the solution.

%
\subsection{Probabilistic and statistical properties of the solution}
Assume that $B$ is a centered Gaussian process with locally $\alpha$-H\"older continuous paths.
\\
\\
The subsection deals with how to transfer probabilistic properties established at Section 3 on the solution of Equation (\ref{main_equation}) to the stochastic process $Z(z_0)$.
\\
\\
In the sequel, the function $F$ satisfies the following assumption.
%

% Assumption : Additional assumptions on the change of variable.

%
\begin{assumption}\label{F_additional_1}
The function $F$ is defined and uniformly continuous on $\mathbb R_+$, satisfies Assumption \ref{change_of_variable_assumption}, and
\begin{displaymath}
\forall x\in\mathbb R_+
\textrm{$,$ }
|F(x)|
\leqslant C(1 + x^k)
\end{displaymath}
with $C > 0$ and $k\in\mathbb N^*$.
\end{assumption}
\noindent
\textbf{Example.} The function $F_{\kappa}$ with $\kappa > 0$ satisfies Assumption \ref{F_additional_1}.
%

% Proposition : Integrability of the solution of the equation driven by a multiplicative noise.

%
\begin{proposition}\label{integrability_solution_multiplicative}
For every $T > 0$, under assumptions \ref{drift_assumption} and \ref{F_additional_1} :
\begin{displaymath}
\|Z(z_0)\|_{\infty,T}\in L^p(\Omega,\mathbb P)
\end{displaymath}
for every $p > 0$.
\end{proposition}
%

% Proof.

%
\begin{proof}
Let $T > 0$ and $t\in [0,T]$ be arbitrarily chosen. By Proposition \ref{solution_estimate} :
\begin{displaymath}
0 < X_t[F^{-1}(z_0)]\leqslant
F^{-1}(z_0) + |b[F^{-1}(z_0)]|T + 2\sigma\|B\|_{\infty,T}.
\end{displaymath}
So, by Jensen's inequality :
\begin{eqnarray*}
 |Z_t(z_0)| & = &
 |F[X_t[F^{-1}(z_0)]]|\\
 & \leqslant &
 C[1 + X_{t}^{k}[F^{-1}(z_0)]]\\
 & \leqslant &
 C(T)(1 +\|B\|_{\infty,T}^{k})
\end{eqnarray*}
with $C(T) > 0$ (deterministic). Therefore, by Fernique's theorem :
\begin{displaymath}
\|Z(z_0)\|_{\infty,T}\in L^p(\Omega,\mathbb P)
\end{displaymath}
for every $p > 0$.
\end{proof}
%

% Proposition : Ergodic theorem for the solution of the equation driven by a multiplicative noise.

%
\begin{proposition}\label{ergodic_theorem_multiplicative}
Assume that $B$ is a two-sided fractional Brownian motion of Hurst parameter $H\in ]0,1[$. Under assumptions \ref{drift_assumption} and \ref{F_additional_1}, for every uniformly continuous function $\psi :\normalfont\textrm I\rightarrow\mathbb R$ with polynomial growth, and every $z_0\in\normalfont\textrm I$,
\begin{displaymath}
\frac{1}{T}\int_{0}^{T}
\psi[Z_t(z_0)]dt
\xrightarrow[T\rightarrow\infty]{}\mathbb E[\psi[F(X^*)]]
\end{displaymath}
almost surely and in $L^p(\Omega,\mathbb P)$ for every $p > 0$.
\end{proposition}
%

% Proof.

%
\begin{proof}
Let an uniformly continuous function $\psi :\textrm I\rightarrow\mathbb R$ with polynomial growth be arbitrarily chosen. Since $F :\mathbb R_+\rightarrow\textrm I$ is also uniformly continuous with polynomial growth, the function $\varphi :=\psi\circ F$ satisfies the conditions of Corollary \ref{ergodic_theorem_CIR}. Then, for every $z_0\in\textrm I$,
\begin{eqnarray*}
 \frac{1}{T}\int_{0}^{T}
 \psi[Z_t(z_0)]dt
 & = &
 \frac{1}{T}\int_{0}^{T}
 \varphi[X_t[F^{-1}(z_0)]]dt\\
 & &
 \xrightarrow[T\rightarrow\infty]{}
 \mathbb E[\varphi(X^*)]
\end{eqnarray*}
almost surely and in $L^p(\Omega,\mathbb P)$ for every $p > 0$. That finishes the proof.
\end{proof}
%

% Proposition : Existence and computation of a density of the solution of the equation driven by a multiplicative noise.

%
\begin{proposition}\label{density_multiplicative}
Let $T > 0$ be arbitrarily fixed. Under assumptions \ref{drift_assumption}, \ref{Gaussian_assumption} and \ref{F_additional_1} with $F^{-1}\in C^1(\overline{\normalfont\textrm I},\mathbb R_+)$, the distribution of $Z_t(z_0)$ has a density with respect to the Lebesgue measure on $(\mathbb R,\mathcal B(\mathbb R))$ for every $t\in ]0,T]$.
\end{proposition}
%

% Proof.

%
\begin{proof}
Let $t\in ]0,T]$ be arbitrarily chosen. By Proposition \ref{Malliavin_derivative_absolute_continuity} ; the distribution of $X_t[F^{-1}(z_0)]$ has a density $\mathbf f_t$ with respect to the Lebesgue measure on $(\mathbb R,\mathcal B(\mathbb R))$. So, by a straightforward application of the transfer theorem :
\begin{equation}\label{density_transfer}
\mathbb P_{Z_t(z_0)}(dz) =
\frac{\mathbf f_t[F^{-1}(z)]}{\dot F[F^{-1}(z)]}dz.
\end{equation}
Therefore, the distribution of $Z_t(z_0)$ has a density with respect to the Lebesgue measure on $(\mathbb R,\mathcal B(\mathbb R))$.
\end{proof}
\noindent
\textbf{Example.} The function $F_{\kappa}$ with $\kappa\in ]0,1]$ satisfies Assumption \ref{F_additional_1} with $F_{\kappa}^{-1}\in C^1(\overline{\textrm I},\mathbb R_+)$.
\\
\\
\textbf{Remark.} Let $t\in ]0,T]$ be arbitrarily chosen. The density $\mathbf f_t$ can be the one provided at Proposition \ref{explicit_density}. So, Equality (\ref{density_transfer}) together with Proposition \ref{explicit_density} provide a density, with a suitable expression, of the distribution of $Z_t(z_0)$.
%

% Proposition : Probabilistic convergence of the approximation.

%
\begin{proposition}\label{probabilistic_convergence_approximation}
Consider $T > 0$ and assume that $B$ is a centered Gaussian process defined on $[0,T]$, with $\alpha$-H\"older continuous paths. Put
\begin{displaymath}
Z_{t}^{n}(z_0) := F[X_{t}^{n}[F^{-1}(z_0)]]
\end{displaymath}
for every $t\in [0,T]$ and $n\in\mathbb N^*$.
\begin{enumerate}
 \item Under assumptions \ref{drift_assumption} and \ref{change_of_variable_assumption} :
 \begin{displaymath}
 \|Z^n(z_0) - Z(z_0)\|_{\infty,T}
 \xrightarrow[n\rightarrow\infty]{\textrm{a.s.}} 0
 \end{displaymath}
 with rate of convergence $\normalfont{\textrm O}(n^{-\alpha})$.
 \item Under assumptions \ref{drift_assumption} and \ref{F_additional_1}, for every $p > 0$,
 \begin{displaymath}
 \sup_{n\in\mathbb N^*}\|Z^n(x_0)\|_{\infty,T}\in
 L^p(\Omega,\mathbb P)
 \end{displaymath}
 and
 \begin{displaymath}
 \lim_{n\rightarrow\infty}\mathbb E[
 \|Z^n(z_0) - Z(z_0)\|_{\infty,T}^{p}] = 0.
 \end{displaymath}
\end{enumerate}
\end{proposition}
%

% Proof.

%
\begin{proof}
Let $p > 0$ be arbitrarily chosen. Put
\begin{displaymath}
X_* :=
\inf_{n\in\mathbb N^*}
\inf_{t\in [0,T]}
X_t[F^{-1}(z_0)]
\wedge X_{t}^{n}[F^{-1}(z_0)]
\end{displaymath}
and
\begin{displaymath}
X^* :=
\sup_{n\in\mathbb N^*}
\sup_{t\in [0,T]}
X_t[F^{-1}(z_0)]\vee
X_{t}^{n}[F^{-1}(z_0)].
\end{displaymath}
Under Assumption \ref{change_of_variable_assumption}, the function $F$ is Lipschitz continuous on $[X_*,X^*]$. So,
\begin{eqnarray*}
 \|Z^n(z_0) - Z(z_0)\|_{\infty,T}
 & \leqslant &
 \|\dot F\|_{\infty,[X_*,X^*]}
 \|X^n[F^{-1}(z_0)] - X[F^{-1}(z_0)]\|_{\infty,T}\\
 & &
 \xrightarrow[n\rightarrow\infty]{\textrm{a.s.}} 0
\end{eqnarray*}
with rate of convergence $\textrm O(n^{-\alpha})$, by Proposition \ref{probabilistic_convergence_Euler}. Under Assumption \ref{F_additional_1}, for every $t\in [0,T]$,
\begin{displaymath}
0 < Z_{t}^{n}(z_0)
\leqslant
C[1 + |X_{t}^{n}[F^{-1}(z_0)]|^k].
\end{displaymath}
Then, by Proposition \ref{probabilistic_convergence_Euler} :
\begin{displaymath}
\sup_{n\in\mathbb N^*}\|Z^n(z_0)\|_{\infty,T}\in
L^p(\Omega,\mathbb P)
\end{displaymath}
for every $p > 0$. So, by Vitali's theorem :
\begin{displaymath}
\lim_{n\rightarrow\infty}
\mathbb E[\|Z^n(z_0) - Z(z_0)\|_{\infty,T}^{p}] = 0.
\end{displaymath}
\end{proof}
\noindent
Assume that $B$ is a fractional Brownian motion of Hurst parameter $H\in ]0,1[$. The values of all the parameters involving in the expressions of $b$ and $F$ are supposed to be known.
\\
\\
As established at Subsection 3.3, $Y(x_0) =\Theta[X(x_0)]$ for every $x_0 > 0$. So,
\begin{displaymath}
Y[F^{-1}(z_0)] =
\Xi[Z(z_0)]
\end{displaymath}
with $\Xi :=\Theta\circ F^{-1}$. Since the parameter $(H,\sigma)$ doesn't involve in the expression of the map $\Xi$, an observation $z(z_0)$ of $Z(z_0)$ provides an observation of $Y[F^{-1}(z_0)]$ by applying the transformation $\Xi$ to $z(z_0)$. Therefore, a consistent estimator of $(H,\sigma)$ for the Ornstein-Uhlenbeck process $Y[F^{-1}(z_0)]$ provides an estimation of the real value of $(H,\sigma)$ from the observation $y[F^{-1}(z_0)] :=\Xi[z(z_0)]$.
\\
The same arguments work on $\Xi^n :=\Theta^n\circ F^{-1}$ applied to $Z^n(z_0)$ instead of $\Xi$ applied to $Z(z_0)$.
%

% Subsection : A fractional Heston model.

%
\subsection{A fractional Heston model}
Financial market models with a fractional stochastic volatility have been already studied in several papers. For instance, the volatility in the Heston model (see Heston \cite{HESTON93}) has been replaced by a fractional process in Comte, Coutin and Renault \cite{CCR12}, taking benefits of its long memory. The subsection deals with another fractional Heston model.
\\
\\
On fractional financial market models, see also Rogers \cite{ROGERS97} and Cheridito \cite{CHERIDITO01}.
\\
\\
Let $T > 0$ be arbitrarily fixed, and assume that $B$ is a fractional Brownian motion of Hurst parameter $H\in ]0,1[$, defined on $[0,T]$. The filtration generated by $B$ is denoted by $\mathbb F := (\mathcal F_t)_{t\in [0,T]}$. By Decreusefond and Ust\"unel \cite{DU99} or Nualart \cite{NUALART06}, Section 5.1.3 ; there exists a unique Brownian motion $B^*$, generating the same filtration $\mathbb F$ than $B$, such that :
\begin{displaymath}
B_t :=
\int_{0}^{t}
K_H(t,s)dB_{s}^{*}
\textrm{ ; }\forall t\in [0,T]
\end{displaymath}
where
\begin{displaymath}
K_H(t,s) :=
\frac{(t - s)^{H - 1/2}}{\Gamma(H + 1/2)}
\textrm F(1/2 - H,H - 1/2,H + 1/2,1 - t/s)\mathbf 1_{[0,t[}(s)
\end{displaymath}
for every $(s,t)\in\mathbb R_{+}^{2}$, and $\textrm F$ is the Gauss hyper-geometric function (see Lebedev \cite{LEBEDEV65}).
\\
\\
Let $\mathbb H^2$ be the space of $\mathbb F$-progressively measurable stochastic processes $(H_t)_{t\in [0,T]}$ such that
\begin{displaymath}
\mathbb E\left(\int_{0}^{T}H_{t}^{2}dt\right) <\infty.
\end{displaymath}
Since $\mathbb F$ is the filtration generated by both $B$ and $B^*$, stochastic processes of $\mathbb H^2$ are integrable with respect to $B^*$ in the sense of It\^o.
\\
\\
\textbf{Notation.} For every $H\in\mathbb H^2$, the It\^o stochastic integral of $H$ with respect to $B^*$ is denoted by
\begin{displaymath}
\left(
\int_{0}^{t}
H_s
\textrm dB_{s}^{*}\right)_{t\in [0,T]}.
\end{displaymath}
Consider the following generalization of the Heston model :
\begin{eqnarray}
 \label{fractional_Heston_prices}
 S_t & = & \displaystyle{S_0 +\int_{0}^{t}\mu_uS_udu +
 \int_{0}^{t}\varphi(Z_u)S_u\textrm dB_{u}^{*}}
 \textrm{ ; }S_0 > 0\\
 \label{fractional_Heston_volatility}
 Z_t & = & \displaystyle{z_0 +\int_{0}^{t}(v - wZ_u)du +
 \zeta\int_{0}^{t}Z_{u}^{\beta}dB_u}
 \textrm{ ; }z_0 > 0
\end{eqnarray}
where $\mu\in C^0([0,T],\mathbb R)$, $v,w > 0$, $\zeta\in\mathbb R^*$, $\beta\in]1 - H,1[$ and $\varphi :\mathbb R_+\rightarrow\mathbb R$ is a continuous function such that :
\begin{equation}\label{polynomial_growth_Heston}
\forall x\in\mathbb R_+
\textrm{, }
|\varphi(x)|\leqslant c(1 + x^n)
\end{equation}
with $c > 0$ and $n\in\mathbb N^*$.
%

% Proposition : Existence and uniqueness of the solution of the generalized Heston model.

%
\begin{proposition}\label{solution_generalized_Heston}
Equation (\ref{fractional_Heston_volatility}) has a unique pathwise solution $Z(z_0)$ such that $Z^{\varphi}(z_0) :=\varphi\circ Z(z_0)\in\mathbb H^2$, and Equation (\ref{fractional_Heston_prices}) has a unique solution $S(z_0)$ in the sense of It\^o such that :
\begin{displaymath}
S_t(z_0) :=
S_0\exp\left[\int_{0}^{t}\left[\mu_s -\frac{1}{2}\varphi^2[Z_s(z_0)]\right]ds +
\int_{0}^{t}\varphi[Z_s(z_0)]\normalfont\textrm dB_{s}^{*}\right]
\end{displaymath}
for every $t\in [0,T]$.
\end{proposition}
%

% Proof.

%
\begin{proof}
Put $x_0 := z_{0}^{1-\beta}$, $\sigma := \zeta(1-\beta)$ and
\begin{displaymath}
b(x) := (1-\beta)(vx^{-\gamma} - wx)
\textrm{ ; }\forall x > 0
\end{displaymath}
where $\gamma :=\beta /(1-\beta)$. By Proposition \ref{existence_solution} together with the rough change of variable formula, the stochastic process $Z(z_0)$ defined by
\begin{displaymath}
Z_t(z_0) := X_{t}^{\gamma + 1}(x_0)
\textrm{ $;$ }
\forall t\in [0,T],
\end{displaymath}
is the unique pathwise solution of Equation (\ref{fractional_Heston_volatility}).
\\
\\
Let $\alpha\in ]0,H[$ be arbitrarily chosen. For every $\omega\in\Omega$ and $t\in [0,T]$, $Z_{t}^{\varphi}(z_0,\omega)$ is the image of $(B_s(\omega))_{s\in [0,t]}$ by the map
\begin{displaymath}
\varphi\circ x_{t}^{\gamma + 1}(x_0,.),
\end{displaymath}
which is continuous from $C^{\alpha}([0,t],\mathbb R)$ into $\mathbb R$ by Proposition \ref{lipschitz_continuity_solution}. So, $Z_{t}^{\varphi}(z_0)$ is $\mathcal F_t$-measurable for every $t\in [0,T]$. In other words, the stochastic process $Z^{\varphi}(z_0)$ is $\mathbb F$-adapted, and even $\mathbb F$-progressively measurable because the paths of $X(x_0)$ are continuous. By Proposition \ref{integrability_solution_multiplicative} together with (\ref{polynomial_growth_Heston}), $Z^{\varphi}(z_0)$ belongs to $\mathbb H^2$.
\\
\\
Therefore, $Z^{\varphi}(z_0)$ is integrable with respect to $B^*$ in the sense of It\^o, and by It\^o's formula (see Revuz and Yor \cite{RY99}, Theorem IV.3.3), the stochastic process $S(z_0)$ defined above is the unique solution, in the sense of It\^o, of Equation (\ref{fractional_Heston_prices}).
\end{proof}
\noindent
According to the usual definition of the Heston model, put $\varphi(x) :=\sqrt x$ for every $x\in\mathbb R_+$.
\\
\\
Consider a financial market consisting of one risky asset of prices process $S(z_0)$ and one risk-free asset of prices function $S^0$, which is the solution of the following ordinary differential equation :
\begin{equation}\label{risk_free_asset_equation}
S_{t}^{0} = S_{0}^{0} +
\int_{0}^{t}r_uS_{u}^{0}du
\end{equation}
where $r\in C^0([0,T],\mathbb R)$. Since $S(z_0)$ is the solution of Equation (\ref{fractional_Heston_prices}) in the sense of It\^o and $S^0$ is the solution of Equation (\ref{risk_free_asset_equation}), by the integration by part formula (see Revuz and Yor \cite{RY99}, Proposition IV.3.1), the actualized prices process $\widetilde S(z_0) := S(z_0)/S^0$ is the solution, in the sense of It\^o, of the following stochastic differential equation :
\begin{displaymath}
\widetilde S_t =
\widetilde S_0 +
\int_{0}^{t}
\sqrt{Z_s(z_0)}\widetilde S_s\textrm dB_{s}^{*}(z_0)
\end{displaymath}
with
\begin{displaymath}
B_{t}^{*}(z_0) :=
\int_{0}^{t}\frac{\mu_s - r_s}{\sqrt{Z_s(z_0)}}ds + B_{t}^{*}
\end{displaymath}
for every $t\in [0,T]$.
%

% References.

%

%
\end{document}